\documentclass[a4paper,10pt,reqno]{amsart}
 
\usepackage[english]{babel}
\usepackage{dsfont} 
\usepackage{graphicx}
\usepackage{color}
\usepackage{hyperref}
\usepackage{cleveref}
\usepackage{enumerate}
\usepackage{amssymb,empheq} 
\usepackage{amsthm} 
\usepackage{xcolor}
\usepackage[foot]{amsaddr}

\allowdisplaybreaks

\newtheorem{thm}{Theorem}[section]
\newtheorem{lem}[thm]{Lemma}

\newtheorem{cor}[thm]{Corollary}

\theoremstyle{definition}

\newtheorem{exm}[thm]{Example}

\theoremstyle{remark}
\newtheorem{rem}[thm]{Remark}

\newcommand{\exmsymbol}{\hfill$\circ$}

\newcommand{\cset}{\mathds{C}}

\newcommand{\nset}{\mathds{N}}
\newcommand{\pset}{\mathds{P}}

\newcommand{\rset}{\mathds{R}}

\newcommand{\tset}{\mathds{T}}
\newcommand{\zset}{\mathds{Z}}

\newcommand{\diff}{\mathrm{d}}
\newcommand{\divv}{\mathrm{div}\,}

\newcommand{\rot}{\mathrm{rot}\,}

\newcommand{\id}{\mathrm{id}}

\newcommand{\cS}{\mathcal{S}}

\newcommand{\cZ}{\mathcal{Z}}

\newcommand{\sA}{\mathsf{A}}

\author{Philipp J.\ di~Dio}

\address{Universit\"at Leipzig, Institut f\"ur Mathematik und Informatik, Augustusplatz 10, D-04109 Leipzig, Germany}
\email{didio@uni-leipzig.de}

\subjclass[2010]{Primary 35Q30; Secondary 76D03, 76D05.}

\keywords{Navier--Stokes equation, periodic, existence, smoothness, Montel space}

\begin{title}[The periodic Navier--Stokes equation: Existence and smoothness]{On the periodic Navier--Stokes equation: An elementary approach to existence and smoothness for all dimensions $n\geq 2$.}
\end{title}

\begin{document}

\begin{abstract}
In this paper we study the periodic Navier--Stokes equation. From the periodic Navier--Stokes equation and the linear equation $\partial_t u = \nu\Delta u + \pset [v\nabla u]$ we derive the corresponding equations for the time dependent Fourier coefficients $a_k(t)$. We prove the existence of a unique smooth solution $u$ of the linear equation by a Montel space version of Arzel\`a--Ascoli. We gain bounds on the $a_k$'s of $u$ depending on $v$. With $v = -u$ these bounds show that a unique smooth solution $u$ of the $n$-dimensional periodic Navier--Stokes equation exists for all $t\in [0,T^*)$ with $T^* \geq 2\nu\cdot \|u_0\|_{\sA,0}^{-2}$. $\|u_0\|_{\sA,0}$ is the sum of the $l^2$-norms of the Fourier coefficients without $e^{i\cdot 0\cdot x}$ of the initial data $u_0\in C^\infty(\tset^n,\rset^n)$ with $\divv u_0=0$. For $\|u_0\|_{\sA,0} \leq \nu$ (small initial data) we get $T^* = \infty$. All results hold for all dimensions $n\geq 2$ and are independent on $n$.
\end{abstract}

\maketitle

\tableofcontents

\section{Introduction}

The dynamics of (incompressible) fluids on $\rset^n$ or in the periodic case on $\tset^n$ (with $\tset := \rset/(2\pi\zset)$ and $n=2,3$) is described by the \emph{Euler} ($\nu = 0$) and \emph{Navier--Stokes} ($\nu > 0$) \emph{equations}
\begin{subequations}\label{eq:ens}
\begin{align}
\partial_t u(x,t) &= \nu\Delta u(x,t) - u\cdot\nabla u(x,t) -\nabla p(x,t) + F(x,t)\\
\divv u(x,t) &= 0
\intertext{with $x\in\rset^n$ or $\tset^n$, $t\geq t_0$ (without loss of generality $t_0=0$), and initial conditions}
u(x,t_0) &= u_0(x).
\end{align}
\end{subequations}
$u(x,t) = (u_1(x_1,\dots,x_n,t),\dots,u_n(x_1,\dots,x_n,t))^T$ is the velocity field of the fluid, $p(x,t)$ is the pressure, and $F(x,t) = (F_1(x_1,\dots,x_n,t),\dots,F_n(x_1,\dots,x_n,t))^T$ are externally applied forces.

The dynamic equation for a compressiable or incompressiable fluid was derived by Leonhard Euler \cite{euler57}. Claude-Louis Navier included viscosity \cite{navier27} and this equation has been rediscovered several times but remained controversial \cite[p.\ 2]{lemari13}. George G.\ Stokes presented a more rigorous derivation \cite{stokes49}. See e.g.\ \cite{lemari13}, \cite{ladyzh03}, and \cite{darrig02} for more on the early history of these equations.

Local existence of a smooth solution was proved by Carl W.\ Oseen \cite{oseen11,oseen27}. In 1933 in his thesis \cite{leray33}, Jean Leray completely solved the two-dimensional case, as Olga A.\ Ladyzhenskaya did put it, by a ``happy circumstance [which] is valid only in the two-dimensional case and only for the Cauchy problem, for periodic boundary conditions, and for a certain special boundary condition'' \cite[p.\ 271]{ladyzh03}. In 1934, Leray showed that Oseen's local solution is unique, exists for all times if the initial data are sufficiently small, and for all initial data a turbulent (what nowadays is called a weak) solution for all times exists \cite{leray34a,leray34b} and therefore introducing the concept of weak solutions and what is now called the Sobolev space $H^1$. Starting in 1953, Ladyzhenskaya rigorously introduced and investigated weak solutions for partial differential equations \cite{ladyzh53}. Eberhard Hopf in 1951 \cite{hopf51} and Andrei A. Kiselev together with Olga Ladyzhenskaya in 1957 \cite{kisele57} used the Galerkin (Fadeo--Galerkin) method for these non-linear equations. In 1969, Ladyzhenskaya finally proved the non-uniqueness of weak (Hopf) solutions \cite{ladyzh69}. In her review paper she summarizes the development as ``seeking solutions of variational problems in spaces dictated by the functional rather than in spaces of smooth functions'' \cite[p.\ 253]{ladyzh03}. Since then the number of function (Sobolev) spaces used or even introduced to treat the Navier--Stokes equation exploded, see e.g.\ \cite{lemari13} for a large collection and discussion of such spaces and results. In 1984, J.\ Thomas Beale, Tosio Kato, and Andrew J.\ Majda proved ``that the maximum norm of the vorticity $\omega(x,t) := \rot u(x,t)$ controls the breakdown of smooth solutions of the 3-D-Euler equations'' \cite{beale84}.

The Euler and the Navier--Stokes equations belong to the most important partial differential equations in pure mathematics and applications. The literature about these equations is enormous, e.g.\ \cite{euler57}, \cite{navier27}, \cite{stokes49}, \cite{oseen11}, \cite{oseen27}, \cite{leray33}, \cite{leray34a}, \cite{leray34b}, \cite{hopf51}, \cite{kisele57}, \cite{ladyzh63}, \cite{caffar82}, \cite{beale84}, \cite{temam84}, \cite{wahl85}, \cite{constan88}, \cite{kreiss89}, \cite{temam95}, \cite{temam01}, \cite{foias01}, \cite{majda02}, \cite{lemarie02}, \cite{darrig02}, \cite{ladyzh03}, \cite{brando04}, \cite{tao06}, \cite{bahour11}, \cite{boyer13}, \cite{lemari13}, \cite{robinson16} to cite only some of the works. This list can be extended to any arbitrary length but however long the list is, global existence on $\rset^3$ or $\tset^3$ of a smooth solution or an example of a finite break down is still open even for $F=0$ \cite{feffer06}.

The main tools to study the Navier--Stokes equation are Sobolev and weak solution theory combined with fixed point theorems. This was introduced by Oseen \cite{oseen11,oseen27} and Leray \cite{leray33,leray34a,leray34b}. In this paper we follow the so called splitting algorithms which are based on the Trotter formula \cite{trotter59}, i.e., a solution of
\[\partial_t u = (A_1  + A_2  + \dots + A_k) u\]
is constructed from the simpler equations $\partial_t u = A_i u$ by
\[u(t) = \lim_{N\to\infty} \left( e^{A_1 t/N}\cdot{\dots}\cdot e^{A_k t/N} \right)^N u_0.\]
Convergence is here usually also controlled by Sobolev methods in suitable Banach spaces. In \cite{didio19ENS} the convergence was not controlled in Banach spaces but in the Montel spaces $\cS(\rset^n)^n$ (Schwartz functions) and $C^\infty(\tset^n)^n$, i.e., bounding one norm exploded in bounding infinitely many semi-norms. In \cite{didio19ENS} we were able to show that if $u_0$ is a Schwartz function then the vorticity $\omega = \rot u$ remains a Schwartz function as long as the classical solution exists. Hence, the time $T$ when $\omega$ leaves $\cS(\rset^n)^n$ is the break down time $T^*$ of the Euler and the Navier--Stokes equation.

In the present work we proceed our Montel space approach. We study the $n$-dimensional periodic Navier--Stokes equation for all $n\geq 2$ more detailed. $C^\infty(\tset^n)$ functions are uniquely determined by their Fourier coefficients. We look at the linear equation
\begin{equation}\label{eq:ens2}
\begin{split}
\partial_t u(x,t) &= \nu\Delta u(x,t) + \pset \big[v(x,t)\cdot\nabla u(x,t) \big]\\
u(x,0) &= u_0(x).
\end{split}
\end{equation}
with known $v\in C^\infty([0,\infty),C^\infty(\tset^n)^n)$ and $\divv v = 0$. It becomes the Euler resp.\ Navier--Stokes equation when $v = -u$. Expanding $u$ and $v$ into Fourier series
\[u(x,t) = \sum_{k\in\zset^n} a_k(t)\cdot e^{i\cdot k\cdot x} \qquad\text{and}\qquad v(x,t) = \sum_{k\in\zset^n} b_k(t)\cdot e^{i\cdot k\cdot x}\]
shows that (\ref{eq:ens2}) is equivalent to the system of ordinary differential equations of the Fourier coefficients $a_k(t)$:
\begin{equation}\label{eq:ak}
\dot{a}_k(t) = -\nu\cdot k^2\cdot a_k(t) + i\cdot \sum_{l\in\zset^n} \langle b_{k-l}(t),l\rangle \cdot \pset_k a_l(t).
\end{equation}
Here, we defined $\pset_0$ and $\pset_k\in\rset^{n\times n}$ for all $k\in\zset^n\setminus\{0\}$ by
\[\pset_0 := \id \qquad\text{and}\qquad \pset_k := \frac{1}{k^2}\begin{pmatrix}
k^2 - k_1^2 & -k_1 k_2 & \cdots & -k_1 k_n\\
-k_1 k_2 & k^2 - k_2^2 & \cdots & -k_2 k_n\\
\vdots & \vdots & \ddots & \vdots\\
-k_1 k_n & -k_2 k_n & \cdots & k^2 - k_n^2
\end{pmatrix},\]
i.e., $\pset_k$ are projections in $\cset^n$ into the orthogonal complement of $k\in\zset^n\subset \cset^n$. The Fourier expansion has previously been investigated by Sobolev methods and especially truncation to finitely many $k$'s ($|k|\leq K$) lead to the Galerkin method, see e.g.\ \cite{constan88}, \cite{kreiss89}, \cite{temam95}, \cite{foias01}, \cite{majda02}. Our approach here is similar but instead of using truncation and the $l^2$-norm of the Fourier coefficient we use
\begin{equation}\label{eq:seminormDef}
\|u(\,\cdot\,,t)\|_{\sA,d}:=\sum_{k\in\zset^n\setminus\{0\}}|k|^d\cdot |a_k(t)|
\end{equation}
for all $d\in\nset_0$ (or even $d\in\rset$) where $|\cdot|$ is the $l^2$-norm in $\cset^n \supseteq\zset^n$.

In the Montel space $C^\infty(\tset^n)$ we need to bound all semi-norms $\|\partial^\alpha f\|_\infty$. But the Leray projector $\pset$ is not compatible with $\|\cdot\|_\infty$ since there are $f\in C^\infty(\tset^n)^n$ such that $\|\pset f\|_\infty > \|f\|_\infty$. Fortunately, the $\|\cdot\|_{\sA,d}$ have the following properties.

\begin{lem}\label{lem:seminormProperties}
For $\|\cdot\|_{\sA,d}$ in (\ref{eq:seminormDef}) and all $f\in C^\infty(\tset^n,\rset^n)$ we have the following:
\begin{enumerate}[(i)]
\item $\|\partial^\alpha f\|_\infty \leq \|f\|_{\sA,d}$ for all $\alpha\in\nset_0^n$ and $d=|\alpha|:= \alpha_1 + \dots + \alpha_n$,

\item $\|\pset f\|_{\sA,d} \leq \|f\|_{\sA,d}$ for all $d\in\rset$,

\item $\|f\|_{\sA,d}\leq \|f\|_{\sA,d+1}$ for all $d\in\rset$,

\item $\|\Delta f\|_{\sA,d} = \|f\|_{\sA,d+2}$ for all $d\in\rset$,

\item $\|\rot f\|_{\sA,d} = \|f\|_{\sA,d+1}$ for all $d\in\zset$ and $f\in C^\infty(\tset^3)^3$ with $\divv f=0$, and

\item $\displaystyle \|f\|_{\sA,d} \leq K_s\cdot \sqrt{\sum_{k\in\zset^n\setminus\{0\}} |k|^{2s+2d}\cdot |f_k|^2}$ with $\displaystyle K_s := \sqrt{\sum_{k\in\zset^n\setminus\{0\}} |k|^{-2s}} <\infty$ for all $s> \frac{n}{2}$.
\end{enumerate}
\end{lem}
\begin{proof}
(i)-(iv) follow immediately from the definition (\ref{eq:seminormDef}), (v) follows from $k\perp f_k$ and $\rot (f_k\cdot e^{i\cdot x\cdot k}) = i\cdot k\times f_k\cdot e^{i\cdot k\cdot k}$, i.e., $|k\times f_k| = |k|\cdot |f_k|$, and (vi) follows from the Cauchy--Schwarz inequality.
\end{proof}

Hence, instead of bounding all $\|\partial^\alpha\cdot\|_\infty$ we will bound all $\|\cdot\|_{\sA,d}$ and (ii) simplifies the calculations immensely. (vi) is the Sobolev imbedding.

The paper is structured as follows: In \Cref{sec:linear} we will prove the $C^\infty(\tset^n,\rset^n)$-valued version of Arzel\`a--Ascoli (\Cref{lem:arzelaascoli}). It is our main tool. In \Cref{thm:main} we then prove the existence of a unique smooth solution $u\in C^d([0,\infty),C^\infty(\tset^n)^n)$ of (\ref{eq:ens2}) and most importantly bound all semi-norms $\|u(\,\cdot\,,t)\|_{\sA,d}$ of the solution in terms of $v$. We denote by $C^d([0,T],C^\infty(\tset^n)^n)$ the set of functions $f:\tset^n\times [0,T]\to\rset^n$ which are $C^d$ in $t\in [0,T]$ and $C^\infty$ in $x\in \tset^n$. Only in \Cref{exm:simpleBreakDown} and \ref{exm:noBreakDown} we deal with complex valued functions and denote these by $C^d([0,T],C^\infty(\tset^n,\cset^n))$. In \Cref{sec:local} we will prove our main result (\Cref{thm:mainNS}). We show that the $n$-dimensional periodic Navier--Stokes equation has a unique smooth solution $u\in C^\infty([0,T^*),C^\infty(\tset^n)^n)$ with the lower bound
\[T^* \geq \frac{2\nu}{\|u_0\|_{\sA,0}^2}.\]
We also gain the smallness condition
\[\|u_0\|_{\sA,0} \leq \nu \qquad\Rightarrow\qquad T^*=\infty\]
and show
\[\|u(\,\cdot\,,t)\|_{\sA,0} < \infty\ \text{for all}\ t\in [0,T] \quad\Rightarrow\quad T^* > T.\]
New compared to other treatments is that the lower bound on $T^*$ depends only on the initial data $\|u_0\|_{\sA,0}$ and $\nu>0$ but is independent on $n$, derivatives of $u_0$, or unknown constants from Sobolev imbeddings or singular integrals. All proofs use only elementary calculations, the Arzel\`a--Ascoli theorem, and the Banach fixed point theorem. While with our approach we win dimension independent results we unfortunately lose the special case $n=2$. We recover $n=2$ by combining the $H^1$-bound of the vorticity and \Cref{thm:mainNS} with \Cref{lem:seminormProperties}(vi). In \Cref{sec:examples} we investigate special explicit examples of (\ref{eq:ens2}) and the growth of the Fourier coefficients.

\section{The unique solution of $\partial_t u = \nu\nabla u + \pset [v\nabla u]$ and its bounds}
\label{sec:linear}

$C^\infty(\tset^n)$ is a complete Montel space with the semi-norms $\|\partial^\alpha f\|_\infty$, $\alpha\in\nset_0^n$. The crucial property we need is that a subset $F\subset C^\infty(\tset^n)$ is bounded if and only if for all $\alpha\in\nset_0^n$ there are constants $C_\alpha>0$ such that $\|\partial^\alpha f\|_\infty \leq C_\alpha$ for all $f\in F$ and $\alpha\in\nset_0^n$. $C^\infty(\tset^n)$ has as a Montel space the Heine--Borel property, i.e., every bounded set $F$ is pre-compact. Since it is complete, $F$ has at least one accumulation point and the Arzel\`a--Ascoli Theorem (\Cref{lem:arzelaascoli}) holds. For more on Montel spaces see e.g.\ \cite{treves67} or \cite{schaef99}. We include the proof of \Cref{lem:arzelaascoli} to make the paper self-contained and since besides \cite{didio19ENS} we are not aware of a reference. The proof follows verbatim the proof in \cite[pp.\ 85--86]{yosida68}.

\begin{lem}[$C^\infty(\tset^n,\rset^n)$-valued version of Arzel\`a--Ascoli]\label{lem:arzelaascoli}
Let $n,m\in\nset$, $T>0$, and $\{f_N\}_{N\in\nset}\subset C([0,T],C^\infty(\tset^n)^m)$. Assume that
\begin{enumerate}[i)]
\item $\sup_{N\in\nset, t\in [0,T]} \|\partial_x^\alpha f_N(x,t)\|_\infty < \infty$ for all $\alpha\in\nset_0^n$, and

\item $\{f_N\}_{N\in\nset}$ is equi-continuous, i.e., for all $\varepsilon>0$ exists $\delta=\delta(\varepsilon) > 0$ such that for all $N\in\nset$ we have
\[|t-s|<\delta \quad\Rightarrow\quad \|f_N(x,t) - f_N(x,s)\|_\infty \leq \varepsilon.\]
\end{enumerate}
Then $\{f_N\}_{N\in\nset}$ is relatively compact in $C([0,T],C^\infty(\tset^n)^m)$.
\end{lem}
\begin{proof}
It is sufficient to prove the result for $m=1$. Then it holds in one component of $f_N$ and by choosing subsequences it holds in all components.

Let $\{t_k\}_{k\in\nset}\subset [0,T]$ be a dense countable subset such that for every $\varepsilon>0$ there is a $k(\varepsilon)\in\nset$ with
\[\sup_{t\in [0,T]} \inf_{1\leq k\leq k(\varepsilon)} |t-t_k| \leq \varepsilon.\]

For any $t\in [0,T]$ the set $\{f_N(\,\cdot\,,t)\}_{N\in\nset}$ is bounded in the complete Montel space $C^\infty(\tset^n)$ and therefore contains a convergent subsequence. Let $(N_{1,i})_{i\in\nset}\subseteq\nset$ be such that $(f_{N_{1,i}}(\,\cdot\,,t_1))_{i\in\nset}$ converges. Take a subsequence $(N_{2,i})_{i\in\nset}$ of $(N_{1,i})_{i\in\nset}$ such that $(f_{N_{2,i}}(\,\cdot\,,t_2))_{i\in\nset}$ converges. By the diagonal process of choice we get a subsequence $(f_{N_i})_{i\in\nset}$ with $N_i := N_{i,i}$ which converges for all $t_k$.

Let $\varepsilon>0$. By the equi-continuity of $\{f_N\}_{N\in\nset}$ there is a $\delta=\delta(\varepsilon)>0$ such that $|t-s|<\delta$ implies $\|f_N(x,t)-f_N(x,s)\|_\infty\leq\varepsilon$. Hence, for every $t\in [0,T]$ there exists a $k$ with $k\leq k(\varepsilon)$ such that
\begin{align*}
\|f_{N_{i}}(x,t) - f_{N_{j}}(x,t)\|_\infty
&\leq \|f_{N_{i}}(x,t) - f_{N_{i}}(x,t_k)\|_\infty + \|f_{N_i}(x,t_k) - f_{N_j}(x,t_k)\|_\infty\\
&\quad + \|f_{N_j}(x,t_k) - f_{N_j}(x,t)\|_\infty\\
&\leq 2\varepsilon + \|f_{N_i}(x,t_k) - f_{N_j}(x,t_k)\|_\infty.
\end{align*}
Thus $\displaystyle\lim_{i,j\to\infty} \sup_{t\in [0,T]} \|f_{N_{i}}(x,t) - f_{N_{j}}(x,t)\|_\infty \leq 2\varepsilon$ and since $\varepsilon>0$ was arbitrary we have $\displaystyle\lim_{i,j\to\infty} \sup_{t\in [0,T]} \|f_{N_{i}}(x,t) - f_{N_{j}}(x,t)\|_\infty =0$. So for every $x\in\tset^n$ the sequence $f_{N_i}(x,\,\cdot\,)$ converges uniformly on $[0,T]$ to a continuous function $f(x,\,\cdot\,)$. Hence by construction $f(\,\cdot\,,t_k)\in C^\infty(\tset^n)$ for all $t_k$ dense in $[0,T]$. But for all $\alpha\in\nset_0^n$
\[\|\partial^\alpha f(x,t)\|_\infty \leq \sup_{s\in [0,T], N\in\nset} \|\partial^\alpha f_N(x,s)\|_\infty < \infty\]
implies $f(\,\cdot\,,t)\in C^\infty(\tset^n)$ for all $t\in [0,T]$, i.e., $f\in C([0,T],C^\infty(\tset^n))$.
\end{proof}

We will now show uniqueness, smoothness, and existence of a solution of (\ref{eq:ens2}).

\begin{thm}\label{thm:main}
Let $n,r\in\nset_0$ with $n\geq 2$, $\nu>0$, $v\in C^r([0,\infty),C^\infty(\tset^n)^n)$ be a periodic real function with $\divv v(x,t) = 0$ for all $t\in [0,\infty)$, and $u_0\in C^\infty(\tset^n)^n$ with $\divv u_0=0$. Then
\begin{align*}
\partial_t u(x,t)&=\nu\Delta u(x,t)+\pset[v(x,t)\nabla u(x,t)]\\
u(x,0) &= u_0(x)
\end{align*}
has a unique smooth solution $u\in C^{r+1}([0,\infty),C^\infty(\tset^n)^n)$. It fulfills
\[\int_{\tset^n} u(x,t)~\diff x = \int_{\tset^n} u_0(x)~\diff x\]
for all $t\in [0,\infty)$ and it fulfills the bounds
\begin{equation}\label{eq:normBound}
\|u(x,t)\|_{\sA,0} \leq \|u_0\|_{\sA,0}\cdot \exp\left(\frac{1}{4\nu}\int_0^t \|v(x,s)\|_{\sA,0}^2~\diff s\right),
\end{equation}
\begin{align}\label{eq:firstBound}
\|u(x,t)\|_{\sA,1} &\leq \|u_0\|_{\sA,1}\cdot
\exp\left(\int_0^t\!\! \|v(x,s)\|_{\sA,1}~\diff s + \frac{1}{4\nu}\int_0^t\!\! \|v(x,s)\|_{\sA,0}^2~\diff s \right),
\end{align}
and for all $d\geq 2$ the bounds
\begin{align}\label{eq:higherBounds}
\|u(x,t)\|_{\sA,d} &\leq \exp\left(\frac{1}{4\nu}\int_0^t \|v(x,s)\|_{\sA,0}^2~\diff s + d\cdot \int_0^t \|v(x,s)\|_{\sA,1}~\diff s \right)\\
&\quad\times\! \left(\! \|u_0\|_{\sA,d} + \sum_{j=2}^d \begin{pmatrix} d\\ j\end{pmatrix}\cdot \int_0^t\!\!\! \|v(x,s)\|_{\sA,j}~\diff s \cdot\!\! \sup_{s\in [0,t]}\!\! \|u(x,s)\|_{\sA,d+1-j}\!\!\right)\notag
\end{align}
for all $t\in [0,\infty)$.

If $\|v(x,t)\|_{\sA,0} + \delta\leq \nu$ for all $t\in [0,\tau]$ with $\tau\geq 0$ and $\delta\geq 0$, then
\begin{equation}\label{eq:normBoundzwei}
\|u(x,t)\|_{\sA,0} \leq \|u_0\|_{\sA,0}\cdot e^{-\delta\cdot t}
\end{equation}
and
\begin{align}\label{eq:firstBoundzwei}
\|u(x,t)\|_{\sA,1} &\leq \|u_0\|_{\sA,1}\cdot
\exp\left(\int_0^t \|v(x,s)\|_{\sA,1}~\diff s -\delta\cdot t\right)
\end{align}
for all $t\in [0,\tau]$.
\end{thm}
\begin{proof}
Let $T>0$. For any $N\in\nset$ take a decomposition $\cZ = \{t_0=0,t_1,\dots, t_N=T\}$ of  $[0,T]$ with $t_0 < t_1 <\dots<t_N$ and $\Delta\cZ := \max_{j=0,\dots,N-1} |t_{j+1}-t_j|$. Since $v$ is in $C^r([0,T],C^\infty(\tset^n)^n)$ we have
\[v(x,t) = \sum_{l\in\zset^n} b_l(t)\cdot e^{i\cdot l\cdot x}\]
where $b_l$ are the Fourier coefficients of $v$.

Let $\{u_N\}_{N\in\nset}$ be a family of functions $u_N:\tset^n\times [0,T]\to\rset$ defined in the following way. Each function $u_N$ is piece-wise on $[t_j,t_{j+1}]$, $j=0,\dots,N-1$, defined by its Fourier coefficients
\begin{align}\label{eq:firstInteration}
a_{k,N}(t) &:= a_{k,N}(t_j)\cdot e^{-\nu\cdot k^2\cdot (t-t_j) + i\cdot\int_{t_j}^t \langle b_0(s),k\rangle~\diff s}\\
&\quad + i\cdot \sum_{l\in\zset^n\setminus\{k\}} \int_{t_j}^{t} \langle b_{k-l}(s),l\rangle~\diff s\cdot \pset_k a_{l,N}(t_j)\cdot e^{-\nu\cdot l^2\cdot (t-t_j) + i\cdot\int_{t_j}^t \langle b_0(s),l\rangle~\diff s}\notag
\end{align}
for all $t\in [t_j,t_{j+1}]$, i.e.,
\begin{equation}\label{eq:secondInteration}
\partial_t a_{k,N}(t) \quad\xrightarrow{t\searrow t_j}\quad -\nu\cdot k^2\cdot a_{k,N}(t_j) + i\cdot \sum_{l\in\zset^n} \langle b_{k-l}(t_j),k\rangle\cdot\pset_k a_{l,N}(t_j).
\end{equation}
(\ref{eq:firstInteration}) is after reordering the following: For each $k\in\zset^n$ the component $a_k(t_j)\cdot e^{i\cdot k\cdot x}$ of $u$ at time $t_j$ becomes at $t\in[t_j,t_{j+1}]$
\begin{align}\label{eq:splitDrei}
&a_{k,N}(t_j)\cdot e^{-\nu\cdot k^2\cdot (t-t_j)}\cdot e^{i\cdot k\cdot x + i\cdot\int_{t_j}^t \langle b_0(s),k\rangle~\diff s}\\
&\quad + i\!\!\sum_{l\in\zset^n\setminus\{0\}} \int_{t_j}^t\langle b_l(s),k\rangle~\diff s\cdot \pset_{k+l} a_{k,N}(t_j)\cdot e^{-\nu\cdot k^2\cdot (t-t_j)+i\cdot\int_{t_j}^t\langle b_0(s),k\rangle\diff s}\cdot e^{i\cdot (k+l)\cdot x}.\notag
\end{align}
Hence, both (\ref{eq:firstInteration}) and (\ref{eq:splitDrei}) contain the same summands. If we sum in the following over the absolute values of the Fourier coefficients both formulations give the same sum, i.e., only a renumbering of the indices appears.

Since $\divv v(x,t)=0$ we have $\langle b_l(s),l\rangle=0$ and hence $a_{0,N}(t)=a_0(0)$ for all $t\in [0,T]$ and $N\in\nset$. Since $b_0(t)\in\rset^n$ for all $t\in [0,T]$ the exponent $i\cdot\int_{t_j}^t \langle b_0(s),k\rangle~\diff s$ induces only a complex rotation, i.e., the absolute values are unchanged. Since we only sum the absolute values, for notational simplicity we let $b_0=0$, i.e., we ignore this imaginary exponent in the following calculations.

In the following we show that for each $d\in\nset_0$ there exists a $C_d > 0$ such that
\[\sup_{\substack{\alpha\in\nset_0^n:|\alpha|=d,\\ N\in\nset,\ t\in [0,T]}} \|\partial^\alpha u_N(x,t)\|_\sA \leq \sup_{t\in [0,T],N\in\nset} \|u_N(x,t)\|_{\sA,d}\leq  C_d\]
to apply \Cref{lem:arzelaascoli}. Since $T>0$ is arbitrary, it is is sufficient to look at $t = T = t_N$.

For $\|u_N(x,t)\|_{\sA,0}$ we get from (\ref{eq:splitDrei}) directly
\begin{align*}
&\phantom{=}\;\;\|u_N(x,t_N)\|_{\sA,0}\\
&= \sum_{k\in\zset^n\setminus\{0\}} |a_{k,N}(t_N)|\\
&\leq \sum_{k\in\zset^n\setminus\{0\}} |a_{k,N}(t_{N-1})|\cdot e^{-\nu\cdot k^2\cdot (t_N-t_{N-1})}\\
&\quad + \sum_{k,l\in\zset^n\setminus\{0\}} \int_{t_{N-1}}^{t_N} \left|\langle b_l(s), k\rangle~\diff s\right|\cdot |a_{k,N}(t_{N-1})|\cdot e^{-\nu\cdot k^2\cdot (t_N-t_{N-1})}\\
&= \sum_{k\in\zset^n\setminus\{0\}} |a_{k,N}(t_{N-1})|\cdot e^{-\nu\cdot k^2\cdot (t_N-t_{N-1})}\cdot \left[ 1 + \sum_{l\in\zset^n\setminus\{0\}} \int_{t_{N-1}}^{t_N} |\langle b_l(s), k\rangle|~\diff s \right]\\
&\leq \sum_{k\in\zset^n\setminus\{0\}}\!\!\!\! |a_{k,N}(t_{N-1})|\cdot \exp\left(\!|k|\int_{t_{N-1}}^{t_N} \|v(x,s)\|_{\sA,0}~\diff s - \nu k^2 (t_N-t_{N-1})\! \right) \tag{\$}
\intertext{If $\|v(x,s)\|_{\sA,0}+\delta\leq \nu$ for all $s\in [0,T]$, then $\|v(x,s)\|_{\sA,0}-\nu\cdot |k|\leq \|v(x,s)\|_{\sA,0}-\nu \leq -\delta$ and we get}
&\leq \sum_{k\in\zset^n\setminus\{0\}} |a_{k,N}(t_{N-1})|\cdot e^{-\delta\cdot (t_N-t_{N-1})}\\
&\ \;\vdots\\
&\leq \sum_{k\in\zset^n\setminus\{0\}} |a_{k,N}(t_0)|\cdot e^{-\delta\cdot (t_N-t_0)} = \|u_0\|_{\sA,0}\cdot e^{-\delta\cdot T}
\intertext{proving (\ref{eq:normBoundzwei}). Assume $\|v(x,s)\|_{\sA,0}\not\leq\nu$, set $K := \nu^{-1}\max_{s\in [t_{N-1},t_N]} \|v(x,s)\|_{\sA,0}$, then for all $k$ with $|k| \geq K$ we have $|k|\int_{t_{N-1}}^{t_N} \|v(x,s)\|_{\sA,0}~\diff s - \nu k^2 (t_{N-1}-t_N)\leq 0$ (contraction) and for $0 < |k| < K$ we have the exponent possibly $>0$ (expansion) with maximum at $\frac{K}{2}$. We devide the sum (\$) into contraction and expansion parts and use the maximum for the expansions (keep this argument in mind, we will use it several times throughout the proof)}
&\leq \sum_{k:|k|\geq K} |a_{k,N}(t_{N-1})| \tag{\#}\\
&\quad + \exp\left(\frac{(t_N-t_{N-1})\max\limits_{s\in [t_{N-1},t_N]} \|v(x,s)\|_{\sA,0}^2}{4\nu} \right)\cdot \sum_{k\neq 0: |k|< K}\!\! |a_k(t_{N-1})|\\
&\leq \exp\left(\frac{(t_N-t_{N-1})\max\limits_{s\in [t_{N-1},t_N]} \|v(x,s)\|_{\sA,0}^2}{4\nu} \right)\cdot \sum_{k\in\zset^n\setminus\{0\}} |a_{k,N}(t_{N-1})|\\
&\ \;\vdots\\
&\leq \exp\left(\frac{1}{4\nu}\sum_{j=1}^N (t_j-t_{j-1})\max\limits_{s\in [t_{j-1},t_j]} \|v(x,s)\|_{\sA,0}^2\right)\cdot \sum_{k\in\zset^n\setminus\{0\}} |a_{k,N}(t_0)|
\intertext{which goes for $N\to\infty$ with $\Delta\cZ\to 0$ to}
&\to \|u_0(x)\|_{\sA,0}\cdot \exp\left(\frac{1}{4\nu} \int_0^T \|v(x,s)\|_{\sA,0}^2~\diff s\right).
\end{align*}
Since the bound converges there exists a $C_0>0$ such that $\|u_N(x,t)\|_{\sA,0}\leq C_0$ for all $t\in [0,T]$ and $N\in\nset$.

For the derivatives ($|\alpha|\geq 1$) it is sufficient to bound $\|u_N(x,t)\|_{\sA,|\alpha|}$. For simplicity we drop the index $N$ in $a_{k,N}$ and only write $a_k$ (but keep in mind that these $a_k$'s still depend on $N$, see (\ref{eq:firstInteration}) and (\ref{eq:splitDrei})).

For the first derivatives $d = |\alpha|=1$ we get from (\ref{eq:splitDrei})
\begin{align*}
&\;\;\;\; \|u_N(x,t_N)\|_{\sA,1} = \sum_{k\in\zset^n\setminus\{0\}} |k|\cdot |a_k(t_N)|\\
&\leq \sum_{k\in\zset^n\setminus\{0\}} |k|\cdot |a_{k,N}(t_{N-1})|\cdot e^{-\nu\cdot k^2\cdot (t_N - t_{N-1})}\\
&\quad + \sum_{k,l\in\zset^n\setminus\{0\}} |k+l|\cdot \int_{t_{N-1}}^{t_N} |\langle b_l(s),k\rangle|~\diff s\cdot |a_k(t_{N-1})|\cdot e^{-\nu\cdot k^2\cdot (t_N-t_{N-1})}\\
&= \sum_{k\in\zset^n\setminus\{0\}}\!\!\!\!\! |a_k(t_{N-1})|\cdot e^{-\nu\cdot k^2\cdot (t_N-t_{N-1})}\cdot \left(|k| + \!\!\!\!\!\sum_{l\in\zset^n\setminus\{0\}}\!\!\!\!\! |k+l| \int_{t_{N-1}}^{t_N} |\langle b_l(s),k\rangle|~\diff s\right)\\
&\leq \sum_{k\in\zset^n\setminus\{0\}} |k|\cdot |a_k(t_{N-1})|\cdot e^{-\nu\cdot k^2\cdot (t_N - t_{N-1})}\cdot \left( 1 + \sum_{l\in\zset^n\setminus\{0\}} \int_{t_{N-1}}^{t_N} |\langle b_l(s),k\rangle|~\diff s\right.\\
&\qquad\qquad \left. + \sum_{l\in\zset^n\setminus\{0\}} |l|\cdot \int_{t_{N-1}}^{t_N} |b_l(s)|~\diff s\right)\\
&\leq \sum_{k\in\zset^n\setminus\{0\}} |k|\cdot |a_k(t_N)|\cdot \exp\left( \int_{t_{N-1}}^{t_N} \|v(x,s)\|_{\sA,1}~\diff s\right.\\
&\qquad\qquad \left. + \int_{t_{N-1}}^{t_N} \|v(x,s)\|_{\sA,0}~\diff s\cdot |k| - \nu\cdot k^2\cdot (t_N-t_{N-1}) \right)\\
&\leq \sum_{k\in\zset^n\setminus\{0\}} |k|\cdot |a_k(t_N)|\cdot \exp\left(\int_{t_{N-1}}^{t_N} \|v(x,s)\|_{\sA,1}~\diff s\right.\tag{§}\\
&\qquad\qquad \left. + |k|\cdot (t_N-t_{N-1})\cdot \left(\max_{s\in [t_{N-1},t_N]} \|v(x,s)\|_{\sA,0} - \nu\cdot |k|\right) \right)
\intertext{If $\|v(x,t)\|_{\sA,0}+ \delta\leq \nu$ for all $t\in [0,T]$, then $\|v(x,t)\|_{\sA,0}-\nu \leq -\delta$ and we get}
&\leq \sum_{k\in\zset^n\setminus\{0\}} |k|\cdot |a_k(t_N)|\cdot \exp\left(\int_{t_{N-1}}^{t_N} \|v(x,s)\|_{\sA,1}~\diff s - \delta\cdot (t_N-t_{N-1})\right)\\
&\ \;\vdots\\
&\leq \sum_{k\in\zset^n\setminus\{0\}} |k|\cdot |a_k(t_0)|\cdot \exp\left(\int_{t_{0}}^{t_N} \|v(x,s)\|_{\sA,1}~\diff s - \delta\cdot (t_N-t_0)\right)\\
&= \|u_0\|_{\sA,1}\cdot \exp\left(\int_0^T \|v(x,s)\|_{\sA,1}~\diff s - \delta\cdot T\right) =:C_1
\intertext{which proves (\ref{eq:firstBoundzwei}). If $\|v(x,t)\|_{\sA,0}+ \delta\not\leq \nu$, then equivalently to the previous case of $\|u_N(x,t)\|_{\sA,0}$ we devide the sum (§) into a contraction part (for $|k|\geq K$) and an expansion part (for $|k|<K$) with $K := \nu^{-1}\cdot\max_{s\in [t_{N-1},t_N]} \|v(x,s)\|_{\sA,0}$ and in the expansion part its maximum is at $K/2$}
&\leq \sum_{k:|k|\geq K} |k|\cdot |a_k(t_{N-1})|\cdot \exp\left(\int_{t_{N-1}}^{t_N} \|v(x,s)\|_{\sA,1}~\diff s\right)\\
&\quad + \sum_{k:0<|k|< K} |k|\cdot |a_k(t_{N-1})|\cdot \exp\left(\int_{t_{N-1}}^{t_N} \|v(x,s)\|_{\sA,1}~\diff s\right.\\
&\qquad\qquad\qquad \left. + \frac{t_N-t_{N-1}}{4\nu}\cdot \max_{s\in [t_{N-1},t_N]} \|v(x,s)\|_{\sA,0}^2 \right)\\
&\leq \exp\left(\int_{t_{N-1}}^{t_N} \|v(x,s)\|_{\sA,1}~\diff s + \frac{t_N-t_{N-1}}{4\nu}\cdot \max_{s\in [t_{N-1},t_N]} \|v(x,s)\|_{\sA,0}^2 \right)\\
&\qquad \times\sum_{k\in\zset^n} |k|\cdot |a_k(t_{N-1})|\\
&\ \;\vdots\\
&\leq \exp\left(\int_{0}^{T} \|v(x,s)\|_{\sA,1}~\diff s + \sum_{j=1}^N \frac{t_j-t_{j-1}}{4\nu}\cdot \max_{s\in [t_{j-1},t_j]} \|v(x,s)\|_{\sA,0}^2 \right)\\
&\qquad \times \sum_{k\in\zset^n} |k|\cdot |a_k(0)|
\intertext{which goes for $N\to\infty$ with $\Delta\cZ\to 0$ to}
&\rightarrow \|u_0\|_{\sA,1}\cdot \exp\left(\int_{0}^{T}\!\! \|v(x,s)\|_{\sA,1}~\diff s + \frac{1}{4\nu}\cdot\int_0^T\!\! \|v(x,s)\|_{\sA,0}^2~\diff s \right)
\end{align*}
and proves (\ref{eq:firstBound}). Since the limit converges the sequence is bounded, i.e., it exists a $C_1>0$ such that
\[\|\partial_j u_N(x,t)\|_\sA \leq \|u_N(x,t)\|_{\sA,1} \leq C_1\]
for all $j=1,\dots,n$, $t\in [0,T]$, and $N\in\nset$.

We have so far proved that for $d=0$ and $1$ there are constants $C_d$ with
\[\|u_N(x,t)\|_{\sA,d}\leq C_d\tag{\%}\]
for all $t\in [0,T]$ and $N\in\nset$. We will now prove (\ref{eq:higherBounds}) and the existence of such constants $C_d$ for all $d\geq 2$ by induction on $d$.

(\ref{eq:firstBound}) is (\ref{eq:higherBounds}) for $d=1$ since then the sum $\sum_{j=2}^d$ is empty. So for $d=1$ (\ref{eq:higherBounds}) is true. So assume for $d\in\nset$ there are constants $C_j$ with $\|u_N(x,t)\|_{\sA,j}\leq C_j$ for all $j=0,\dots,d-1$, $t\in [0,T]$ and $N\in\nset$. We show then also $C_d$ exists.

From (\ref{eq:splitDrei}) we get
\begin{align*}
&\phantom{=}\;\, \|u_N(x,t_N)\|_{\sA,d}\\
&= \sum_{k\in\zset^n\setminus\{0\}} |k|^d\cdot |a_k(t_N)|\\
&\leq \sum_{k\in\zset^n\setminus\{0\}} |k|^d\cdot |a_k(t_{N-1})|\cdot e^{-\nu\cdot k^2\cdot (t_N - t_{N-1})}\\
&\quad + \sum_{k,l\in\zset^n\setminus\{0\}} |k+l|^d\cdot \int_{t_{N-1}}^{t_N} |\langle b_l(s),k\rangle|~\diff s\cdot |a_k(t_{N-1})|\cdot e^{-\nu\cdot k^2\cdot (t_N-t_{N-1})}\\
&\leq \sum_{k\in\zset^n\setminus\{0\}} |k|^d\cdot |a_k(t_{N-1})|\cdot e^{-\nu\cdot k^2\cdot (t_N - t_{N-1})}\\
&\quad + \sum_{k,l\in\zset^n\setminus\{0\}} |k|^d\cdot \int_{t_{N-1}}^{t_N} |\langle b_l(s),k\rangle|~\diff s\cdot |a_k(t_{N-1})|\cdot e^{-\nu\cdot k^2\cdot (t_N-t_{N-1})} \\
&\quad + d\cdot\sum_{k,l\in\zset^n\setminus\{0\}} |k|^d\cdot |l| \int_{t_{N-1}}^{t_N} |b_l(s)|~\diff s\cdot |a_k(t_{N-1})|\cdot e^{-\nu\cdot k^2\cdot (t_N-t_{N-1})} \\
&\quad + \sum_{k,l\in\zset^n\setminus\{0\}} \sum_{j=2}^{d} \begin{pmatrix} d\\ j\end{pmatrix}\cdot |k|^{d-j}\cdot |l|^j\cdot |k|\cdot \int_{t_{N-1}}^{t_N} |b_l(s)|~\diff s\cdot |a_k(t_{N-1})|\cdot e^{-\nu\cdot k^2\cdot (t_N - t_{N-1})}\\
&\leq \sum_{k\in\zset^n\setminus\{0\}} |k|^d\cdot |a_k(t_{N-1})|\cdot e^{-\nu\cdot k^2\cdot (t_N-t_{N-1})}\cdot \left( 1 + |k|\cdot\int_{t_{N-1}}^{t_N} \|v(x,s)\|_{\sA,0}~\diff s \right.\\
&\qquad\qquad \left. + d\cdot \int_{t_{N-1}}^{t_N} \|v(x,s)\|_{\sA,1}~\diff s \right)\\
&\quad + \sum_{j=2}^d \begin{pmatrix} d\\ j\end{pmatrix}\cdot \int_{t_{N-1}}^{t_N} \|v(x,s)\|_{\sA,j}~\diff s \cdot \|u_N(t_{N-1})\|_{\sA,d+1-j}\\
&\leq \sum_{k\in\zset^n\setminus\{0\}} |k|^d\cdot |a_k(t_{N-1})|\cdot \exp\left( |k|\cdot \int_{t_{N-1}}^{t_N} \|v(x,s)\|_{\sA,0}~\diff s - \nu\cdot k^2\cdot (t_N-t_{N-1})\right.\\
&\qquad\qquad \left. + d\cdot \int_{t_{N-1}}^{t_N} \|v(x,s)\|_{\sA,1}~\diff s \right)\\
&\quad + \sum_{j=2}^d \begin{pmatrix} d\\ j\end{pmatrix}\cdot \int_{t_{N-1}}^{t_N} \|v(x,s)\|_{\sA,j}~\diff s\cdot \sup_{s\in [0,T]} \|u_N(x,s)\|_{\sA,d+1-j}\\
&\leq \exp\left(\frac{t_N-t_{N-1}}{4\nu}\max_{s\in [t_{N-1},t_N]} \|v(x,s)\|_{\sA^2,0} + d\cdot \int_{t_{N-1}}^{t_N} \|v(x,s)\|_{\sA,1}~\diff s \right)\\
&\qquad\times \sum_{k\in\zset^n} |k|^d \cdot |a_k(t_{N-1})|\\
&\quad + \sum_{j=2}^d \begin{pmatrix} d\\ j\end{pmatrix}\cdot \int_{t_{N-1}}^{t_N} \|v(x,s)\|_{\sA,j}~\diff s\cdot \sup_{s\in [0,T]} \|u_N(x,s)\|_{\sA,d+1-j}\\
&\ \;\vdots\\
&\leq \exp\left(\sum_{j=1}^N\frac{t_j-t_{j-1}}{4\nu}\max_{s\in [t_{j-1},t_j]} \|v(x,s)\|_{\sA,0}^2 + d\cdot \int_{t_0}^{t_N} \|v(x,s)\|_{\sA,1}~\diff s \right)\\
&\quad\times \left( \|u_0\|_{\sA,d} + \sum_{j=2}^d \begin{pmatrix} d\\ j\end{pmatrix}\cdot  \int_{t_0}^{t_N} \|v(x,s)\|_{\sA,1}~\diff s\right)\cdot \sup_{s\in [0,T]} \|u_N(x,s)\|_{\sA,d+1-j}
\intertext{The bound depends only on $\|u_N(x,s)\|_{\sA,j}$ for $j=0,1,\dots,d-1$ (all converge by induction hypothesis) and hence the expression goes for $N\to\infty$ with $\Delta\cZ\to 0$ to}
&\to \exp\left(\frac{1}{4\nu}\int_0^T \|v(x,s)\|_{\sA,0}^2~\diff s + d\cdot \int_0^T \|v(x,s)\|_{\sA,1}~\diff s \right)\\
&\quad\times \left( \|u_0\|_{\sA,d} + \sum_{j=2}^d \begin{pmatrix} d\\ j\end{pmatrix}\cdot  \int_0^T\!\!\! \|v(x,s)\|_{\sA,j}~\diff s \cdot \sup_{s\in [0,T]} \|u(x,s)\|_{\sA,d+1-j}\right)
\end{align*}
which proves (\ref{eq:higherBounds}) for $d$. Since the limit exists the sequence is bounded, i.e., it exists a constant $C_d$ such that (\%) holds. Hence, by induction (\ref{eq:higherBounds}) and (\%) hold for all $d\in\nset$ with $d\geq 2$.

We apply now \Cref{lem:arzelaascoli}. We just proved
\[\max_{\alpha\in\nset_0^n:|\alpha|=d} \|\partial^\alpha u_N(x,t)\|_\infty \leq \sup_{s\in [0,T],N\in\nset} \|u_N(x,s)\|_{\sA,d} \leq C_d <\infty\]
for all $d\in\nset$, i.e., condition (i) of \Cref{lem:arzelaascoli} is fulfilled for the family $\{u_N\}_{N\in\nset}$. Since all $u_N$ are piece-wise differentiable in $t$ with (\ref{eq:secondInteration}) and all derivatives of $\partial^\alpha u_N$ are bounded, the family $\{u_N\}_{N\in\nset}$ is Lipschitz in $t$ with a Lipschitz constant $L$ independent on $N\in\nset$ and $t\in [0,T]$, i.e., $\{u_N\}_{N\in\nset}$ is equi-continuous and condition (ii) of \Cref{lem:arzelaascoli} is fulfilled. Hence, by \Cref{lem:arzelaascoli} the family $\{u_N\}_{N\in\nset}$ is relatively compact and there exists an accumulation point $u\in C([0,T],C^\infty(\tset^n)^n)$ and a subsequence $\{N_j\}_{j\in\nset}\subseteq\nset$ such that for all $\alpha\in\nset_0^n$ we have
\[\sup_{t\in [0,T]} \|\partial^\alpha u_{N_j}(x,t) - \partial^\alpha u(x,t)\|_\infty \xrightarrow{j\to\infty} 0.\tag{$*$}\]
By (\ref{eq:secondInteration}) $u$ is a smooth solution of the initial value problem with the desired bounds for all $t\in [0,T]$. Since $a_{0,N}(t) = a_0(0)$ for all $t\in [0,T]$ is constant, so is $a_0(t)$ of $u$:
\[\int_{\tset^n} u(x,t)~\diff x = a_0(t) = a_0(0) = \int_{\tset^n} u_0(x)~\diff x.\]

We now show that $u$ is $C^{r+1}$ in $t$. By (\ref{eq:firstInteration}) $\partial_t u_N$ is Riemann integrable for any fixed $x\in\tset^n$, by ($*$) all derivatives of $u_{N_j}$ converge uniformly on $[0,T]\times \tset^n$, and therefore we have
\begin{align*}
u(x,t) &= \lim_{j\to\infty} u_{N_j}(x,t)\\
&= \lim_{j\to\infty} \int_0^t \partial_t u_{N_j}(x,s)~\diff s\\
&\overset{(*)}{=} \int_0^t \lim_{j\to\infty} \partial_t u_{N_j}(x,s)~\diff s\\
&= \int_0^t \underbrace{\nu\Delta u(x,t) + \pset [v(x,s)\nabla u(x,s)]}_{\text{integrand (I)}}~\diff s.
\end{align*}
Since $u\in C([0,T],C^\infty(\tset^n)^n)$ the integrand (I) is $C^0$ in $t$. Hence, $u$ is the integral of a $C^0$-function, therefore $C^1$ in $t$, and (I) is $C^1$ in $t$. Proceeding this arguments shows that the integrand (I) is $C^r$ in $t$ since $v\in C^r$ and hence $u\in C^{r+1}([0,T],C^\infty(\tset^n)^n)$.

Since $T>0$ was arbitrary, we can set $T=1$ and from $[0,\infty) = \bigcup_{j\in\nset_0} [j,j+1]$ it follows that $u$ exists for all times $t\in [0,\infty)$.

The uniqueness follows from the standard $L^2$-norm estimate since $\divv v=0$: Let $u$ and $\tilde{u}$ be two smooth solutions with initial data $u_0$. Then
\begin{align*}
\langle u-\tilde{u},u-\tilde{u}\rangle_{L^2(\tset^n)} &\geq 0
\intertext{with}
\langle u-\tilde{u},u-\tilde{u}\rangle_{L^2(\tset^n)}\big|_{t=0} &= 0
\intertext{and}
\partial_t \langle u-\tilde{u},u-\tilde{u}\rangle_{L^2(\tset^n)} &= 2\langle u-\tilde{u}, \nu\Delta (u-\tilde{u}) + \pset [v\nabla (u-\tilde{u})]\rangle_{L^2(\tset^n)}\\
&= -2\nu\cdot \|\nabla (u-\tilde{u})\|_{L^2(\tset^n)}^2 + 2\langle u-\tilde{u},v\nabla (u-\tilde{u})\rangle_{L^2(\tset^n)}\\
&= -2\nu\cdot \|\nabla (u-\tilde{u})\|_{L^2(\tset^n)}^2\\
&\leq 0
\end{align*}
since with $w := u-\tilde{u}$ and $\divv v=0$ we have
\begin{multline*}
\langle w,v\nabla w\rangle_{L^2(\tset^n)} = \int_{\tset^n} \sum_{i,j=1}^n w_i\cdot v_j \cdot \partial_j w_i~\diff x
= \frac{1}{2}\int_{\tset^n} \sum_{i,j=1}^n v_j\cdot \partial_j w_i^2~\diff x\\
= -\frac{1}{2} \int_{\tset^n} \sum_{i,j=1}^n w_i^2\cdot \partial_j v_j~\diff x
= -\frac{1}{2} \int_{\tset^n} w^2\cdot \divv v~\diff x
= 0.
\end{multline*}
Hence, $\langle u-\tilde{u},u-\tilde{u}\rangle_{L^2(\tset^n)} = 0$ and therefore $u = \tilde{u}$ for all $t\in [0,\infty)$.
\end{proof}

Note, since $u$ is unique also the accumulation point of $\{u_N\}_{N\in\nset}$ is unique, i.e., $u_N$ converges to $u$ in $C^r([0,T],C^\infty(\tset^n)^n)$:
\[ \sup_{(x,t)\in \tset^n\times [0,T]} |\partial^\alpha u_N(x,t) - \partial^\alpha u(x,t)| \to 0\]
for all $\alpha\in\nset_0^n$ and $T>0$.

\section{Local existence and uniqueness for the periodic Navier--Stokes equation}
\label{sec:local}

For a better reading we split the proof of our main result in \Cref{thm:mainNS} into two parts. The first part is \Cref{lem:existence}. We use there \Cref{thm:main} with the bounds (\ref{eq:normBound}) - (\ref{eq:firstBoundzwei}) to show preliminary lower bounds on $T^*$ depending on $\|u_0\|_{\sA,0}$ and $\|u_0\|_{\sA,1}$. In \Cref{thm:mainNS} we then remove the dependency on $\|u_0\|_{\sA,1}$ and prove uniqueness.

\begin{lem}\label{lem:existence}
Let $n\in\nset$, $n\geq 2$, $\nu>0$ and $u_0\in C^\infty(\tset^n,\rset^n)$ with $\divv u_0=0$. There exists
\begin{equation*}
T^* \geq \left.\begin{cases}
\infty & \text{if}\ \|u_0\|_{\sA,1}=0,\\
\frac{2\nu}{\|u_0\|_{\sA,0}^2} & \text{if}\ \|u_0\|_{\sA,0}^2 \geq 4\cdot\nu\cdot\|u_0\|_{\sA,1},\\
\frac{1}{\|u_0\|_{\sA,1}} - \frac{\|u_0\|_{\sA,0}^2}{8\nu\cdot \|u_0\|_{\sA,1}^2} - \frac{3\nu}{2\|u_0\|_{\sA,0}^2} & \text{else}\end{cases}\right\} > 0.
\end{equation*}
such that the $n$-dimensional periodic Navier--Stokes equation has a smooth solution $u\in C^\infty([0,T^*),C^\infty(\tset^n)^n)$. $u$ fulfills the bounds (\ref{eq:normBound}) - (\ref{eq:firstBoundzwei}) with $u=-v$.

Additionally, the a priori bound $\|u(x,t)\|_{\sA,1} \leq C < \infty$ for all $t\in [0,T]$ for some $T>0$ implies $u\in C^\infty([0,T],C^\infty(\tset^n)^n)$.
\end{lem}
\begin{proof}
For $\|u_0\|_{\sA,1} = 0$ we have that $u_0(x)$ is constant and hence $u(x,t) = u_0(x)$ is a smooth solution for all times $t\geq 0$. Therefore, we can assume w.l.o.g.\ $\|u_0\|_{\sA,1}>0$.

For $\varepsilon\in (0,1]$ let $u^{(\varepsilon)}$ be the unique time-delayed solution from \Cref{thm:main}, i.e., the unique solution of
\[\partial_t u^{(\varepsilon)}(x,t) = \nu\Delta u^{(\varepsilon)}(x,t) - \pset[u^{(\varepsilon)}(x,t-\varepsilon)\nabla u^{(\varepsilon)}(x,t)]\]
with $u^{(\varepsilon)}(x,t) = u_0(x)$ for all $t\in [-\varepsilon,0]$. This unique solution exists since for $t\in [0,\varepsilon]$ we have $v(x,t) = u^{(\varepsilon)}(x,t-\varepsilon) = u_0(x)$ in \Cref{thm:main}, i.e., $u^{(\varepsilon)}$ exists uniquely for $t\in [0,\varepsilon]$ and hence $v$ exists for $[0,2\varepsilon]$. Proceeding this argument we see that $u^{(\varepsilon)}$ exists uniquely for all $t\in [0,\infty)$ for any fixed $\varepsilon>0$.

We want to apply \Cref{lem:arzelaascoli} to the family $\{u^{(\varepsilon)}\}_{\varepsilon\in (0,1]}$. Hence, we want to determine the interval $I\subseteq [0,\infty)$ where we have
\[C_d(t) := \limsup_{\varepsilon\to 0} \|u^{(\varepsilon)}(x,t)\|_{\sA,d} <\infty\]
for $t\in I$ and all $d\in\nset_0$.

From \Cref{thm:main} (\ref{eq:normBound}) with $\|v(x,t)\|_{\sA,0} = \|u(x,t-\varepsilon)\|_{\sA,0}$ we find for $\varepsilon\to 0$
\[C_0(t) \leq \|u_0\|_{\sA,0}\cdot \exp\left(\frac{1}{4\nu} \int_0^t C_0(s)^2~\diff s \right).\]
In the worst case we have equality. Differentiating this equality gives
\begin{align}\label{eq:f0}
f_0(t) = f_0(0)\cdot \exp\left( \frac{1}{4\nu} \int_0^t f_0(s)^2~\diff s \right)\quad
&\overset{\frac{\diff}{\diff t}}{\Rightarrow}\quad f'_0(t) = \frac{1}{4\nu}\cdot f_0^3(t)\\
&\Rightarrow\quad C_0(t)\leq f_0(t) = \sqrt{\frac{2\nu}{2\nu\|u_0\|_{\sA,0}^{-2}- t}}\notag
\end{align}
with singularity at $T_0 := 2\nu\cdot \|u_0\|_{\sA,0}^{-2}>0$. For $d=1$ we find from the bound (\ref{eq:firstBound})
\begin{align*}
&f_1(t) = f_1(0)\cdot \exp\left(\int_0^t f_1(s)~\diff s + \frac{1}{4\nu} \int_0^t f_0(s)^2~\diff s\right)\\
\overset{\frac{\diff}{\diff t}}{\Rightarrow}\quad &f'_1(t) = f_1(t)\cdot \left(f_1(t) + \frac{f_0(t)^2}{4\nu}\right) = f_1(t)\cdot \left(f_1(t) + \frac{1}{4\nu\cdot \|u_0\|_{\sA,0}^{-2} - 2t} \right)\\
\Rightarrow\quad & C_1(t)\leq f_1(t) = \frac{1}{\gamma\cdot \sqrt{4\nu\cdot \|u_0\|_{\sA,0}^{-2}-2t} + (4\nu\cdot \|u_0\|_{\sA,0}^{-2} - 2t)}\\
&\phantom{C_1(t)\leq f_1(t)} =\frac{1}{\sqrt{c-2t}\cdot(\gamma + \sqrt{c-2t})}.
\end{align*}
with $\gamma := \frac{\|u_0\|_{\sA,0}}{2\cdot\sqrt{\nu}\cdot\|u_0\|_{\sA,1}} - \frac{2\sqrt{\nu}}{\|u_0\|_{\sA,0}}$ and $c = 4\nu\cdot \|u_0\|_{\sA,0}^{-2}$.

If $\gamma\geq 0$, then $f_1$ (and $f_0$) exists on $[0,T_0)$ with a singularity at $T_0$.

If $\gamma\in (-\sqrt{c},0)$, then $\gamma + \sqrt{c-2t} \geq 0$ for all $t\in [0,T_1]$ with $T_1:= c-\gamma^2 = 2\nu\|u_0\|_{\sA,0}^{-2} - \gamma^2 = \frac{1}{\|u_0\|_{\sA,1}} - \frac{\|u_0\|_{\sA,0}^2}{8\nu\cdot \|u_0\|_{\sA,1}^2} - \frac{3\nu}{2\|u_0\|_{\sA,0}^2}$ and $f_1$ has a singularity at $T_1\in (0,T_0)$. $f_0$ and $f_1$ exists both on $[0,T_1)$. Which proves the third case.

Let $[0,\tau]\subset [0,T)$ where $T$ is one of the three cases $\infty$, $T_0$, or $T_1$ we just calculated. Then $f_0(t)$ and $f_1(t)$ are continuous on $[0,\tau]$ and therefore bounded. $C_d(t)$ fulfills the bounds (\ref{eq:higherBounds}) for all $d\geq 2$. But (\ref{eq:higherBounds}) has the structure
\[C_d(t) \leq B_d(t) + \tilde{B}_d(t)\cdot \int_0^t C_d(s)~\diff s\]
where $B_d$ and $\tilde{B}_d$ are continuous non-decreasing functions depending only on $C_0$, $\dots$, $C_{d-1}$. Hence,
\[C_d(t) \leq B_d(\tau) + \tilde{B}_d(\tau)\cdot \int_0^t C_d(s)~\diff s\]
with equality in the worst cast. Differentiating the worst cast gives
\[f_d'(t) = \tilde{B}_d(\tau)\cdot f_d(t) \quad\Rightarrow\quad C_d(t) \leq f_d(t) = B_d(\tau)\cdot\exp\left(\tilde{B}_d(\tau)\cdot t\right),\]
i.e., all $C_d$ are bounded functions on $[0,\tau]\subset [0,T)$ with $T>0$.

Since $C_d(t)<\infty$ for all $d\in\nset$ and $t\in [0,\tau]$ the family $\{u^{(\varepsilon)}\}_{\varepsilon\in (0,1]}$ fulfills condition (i) of \Cref{lem:arzelaascoli}. But all $u^{(\varepsilon)}$ fulfill (\ref{eq:ens2}) with $v(x,t) = u^{(\varepsilon)}(x,t-\varepsilon)$, i.e.,
\[\sup_{s\in [0,\tau],\varepsilon\in(0,1]} \|\partial_t u^{(\varepsilon)}(x,s)\|_\infty <\infty,\]
and condition (ii) of \Cref{lem:arzelaascoli} is fulfilled. Therefore, the family $\{u^{(\varepsilon)}\}_{\varepsilon\in (0,1]}$ is relatively compact and has an accumulation point $u\in C([0,\tau],C^\infty(\tset^n)^n)$. Let $(\varepsilon_j)_{j\in\nset}\subset (0,1]$ with $\varepsilon_j\to 0$ as $j\to\infty$. Then
\[\sup_{t\in [0,\tau]} \|\partial^\alpha u^{(\varepsilon_j)}(x,t)- \partial^\alpha u(x,t)\|_\infty\xrightarrow{j\to\infty} 0 \tag{$*$}\]
for all $\alpha\in\nset_0^n$. Similarly as in the proof of \Cref{thm:main} we find that
\begin{align*}
u(x,t) &= \lim_{j\to\infty} u^{(\varepsilon_j)}(x,t)\\
&= \lim_{j\to\infty} \int_0^t \partial_t u^{(\varepsilon_j)}(x,s)~\diff s\\
&\overset{(*)}{=} \int_0^t \lim_{j\to\infty} \partial_t u^{(\varepsilon_j)}(x,s)~\diff s\\
&= \int_0^t \nu\Delta u(x,t) + \pset [u(x,s)\nabla u(x,s)]~\diff s
\end{align*}
implies that $u\in C^\infty([0,\tau],C^\infty(\tset^n)^n)$ solves the Navier--Stokes equation. But since $[0,\tau]\subset [0,T)\subseteq [0,T^*)$ was arbitrary we have a solution $u\in C^\infty([0,T^*),C^\infty(\tset^n)^n)$.

Since $u$ fulfills the bounds (\ref{eq:normBound}) - (\ref{eq:firstBoundzwei}) with $u=-v$ we find that as long as $\|u(x,t)\|_{\sA,1}$ is bounded, so are $\|u(x,t)\|_{\sA,k}$ for all $k\in\nset_0$.
\end{proof}

The preliminary lower bounds on $T^*$ in \Cref{lem:existence} depend on $\|u_0\|_{\sA,0}$ and $\|u_0\|_{\sA,1}$. We remove now the dependency on $\|u_0\|_{\sA,1}$ and prove uniqueness.

\begin{thm}\label{thm:mainNS}
Let $n\in\nset$ with $n\geq 2$, $\nu>0$, and $u_0\in C^\infty(\tset^n)^n$ with $\divv u_0=0$. There exists
\begin{equation}\label{eq:timeStar}
T^* \geq \begin{cases} \infty & \text{if}\quad \|u_0\|_{\sA,0}\leq \nu,\ \text{or}\\ \frac{2\nu}{\|u_0\|_{\sA,0}^2} & \text{if}\quad \|u_0\|_{\sA,0}>\nu\end{cases}
\end{equation}
such that the $n$-dimensional periodic Navier--Stokes equation has a unique smooth solution $u\in C^\infty([0,T^*),C^\infty(\tset^n)^n)$.

If the a priori bound $\|u(x,t)\|_{\sA,0}\leq C<\infty$ holds for all $t\in [0,T]$, then $T^* > T$.
\end{thm}
\begin{proof}
At first we prove the second statement, i.e., the a priori bound $\|u_0(x,t)\|_{\sA,0} \leq C < \infty$ for all $t\in [0,T]$ implies a smooth solution for all $t\in [0,T]$.

By \Cref{lem:existence} there exists a time $T^*\in [0,\infty)$ such that a smooth solution $u(x,t)$ exists for all $t\in [0,T^*)$. Set $v(x,t) = -u(x,t)$ in \Cref{thm:main} and let $T\in [0,T^*)$. Then for $N\to\infty$ with $\Delta\cZ_N\to 0$ a subsequence of the approximate solutions $u_N$ in (\ref{eq:firstInteration}) converges to $u$ with respect to $\|\cdot\|_{\sA,d}$ for all $d\in\nset_0$ and $t\in [0,T]$. We show that for all $k\in\zset^n$ there are constants $A_k>0$ such that $|a_k(t)|\leq A_k$ and $\sum_{k\in\zset^n} A_k < \infty$. We do so by following the proof of the bound (\ref{eq:normBound}) from \Cref{thm:main}, i.e., bounding $\|u_N(x,t_N)\|_{\sA,0}$ by the calculation
\begin{align*}
&\phantom{=}\;\;\|u_N(x,t_N)\|_{\sA,0}\\
&= \sum_{k\in\zset^n\setminus\{0\}} |a_{k,N}(t_N)|\\
&\leq \sum_{k\in\zset^n\setminus\{0\}} |a_{k,N}(t_{N-1})|\cdot e^{-\nu\cdot k^2\cdot (t_N-t_{N-1})} \tag{$*$}\\
&\quad + \sum_{k,l\in\zset^n\setminus\{0\}} \int_{t_{N-1}}^{t_N} \left|\langle b_l(s), k\rangle\right|~\diff s\cdot |a_{k,N}(t_{N-1})|\cdot e^{-\nu\cdot k^2\cdot (t_N-t_{N-1})}\\
&= \sum_{k\in\zset^n\setminus\{0\}} |a_{k,N}(t_{N-1})|\cdot e^{-\nu\cdot k^2\cdot (t_N-t_{N-1})}\cdot \left[ 1 + \sum_{l\in\zset^n\setminus\{0\}} \int_{t_{N-1}}^{t_N} |\langle b_l(s), k\rangle|~\diff s \right]\\
&\leq \sum_{k\in\zset^n\setminus\{0\}}\!\!\!\! |a_{k,N}(t_{N-1})|\cdot \exp\left(\!|k|\int_{t_{N-1}}^{t_N} \|v(x,s)\|_{\sA,0}~\diff s - \nu k^2 (t_N-t_{N-1})\! \right)\\
&\leq \exp\left(\frac{(t_N-t_{N-1})\max\limits_{s\in [t_{N-1},t_N]} \|v(x,s)\|_{\sA,0}^2}{4\nu} \right)\cdot \sum_{k\in\zset^n\setminus\{0\}} |a_{k,N}(t_{N-1})|\\
&\ \;\vdots\\
&\leq \exp\left(\frac{1}{4\nu}\sum_{j=1}^N (t_j-t_{j-1})\max\limits_{s\in [t_{j-1},t_j]} \|v(x,s)\|_{\sA,0}^2\right)\cdot \sum_{k\in\zset^n\setminus\{0\}} |a_{k,N}(t_0)|\\
&\xrightarrow{N\to\infty\ \text{with}\ \Delta\cZ_N\to 0} \|u_0(x)\|_{\sA,0}\cdot \exp\left(\frac{1}{4\nu} \int_0^{T} \|v(x,s)\|_{\sA,0}^2~\diff s\right)
\end{align*}
But in each step ($*$) we sum over the absolute values of the Fourier coefficients in (\ref{eq:firstInteration}) separately and by rearranging (\ref{eq:firstInteration}) to (\ref{eq:splitDrei}) also the absolute values of the Fourier coefficients are rearranged without changing the bound on $\|u_N(x,t_N)\|_{\sA,0}$ and $\|u(x,t)\|_{\sA,0}$. Hence, for each $a_k$ we get
\[|a_k(t)|\leq A_k \quad\text{with}\quad \sum_{k\in\zset^n} A_k = C <\infty. \tag{\&}\]
for all $t\in [0,T]$. Note that $a_0(t)=a_0(0)$ is constant and therefore
\[\sum_{k\in\zset^n\setminus\{0\}} \max_{t\in [0,T]} |a_k(t)| \leq \sum_{k\in\zset^n\setminus\{0\}}\!\!\! A_k \leq \|u_0(x)\|_{\sA,0}\cdot \exp\left(\frac{1}{4\nu} \int_0^{T} \|v(x,s)\|_{\sA,0}^2~\diff s\right).\]

We now show that the $A_k$'s imply that also $\|u(x,t)\|_{\sA,1}$ is bounded. The $a_k$ fulfill the coefficient formulation (\ref{eq:ak}) of the Navier--Stokes equation with $b_k = -a_k$. For each $a_k$ it is an inhomogeneous first order ODE $f'(t) = c\cdot f(t) + g(t)$ solved by
\[f(t) = \int_0^t g(s)\cdot e^{-c\cdot s}~\diff s\cdot e^{c\cdot t} + f(0)\cdot e^{c\cdot t}.\]
Hence, (\ref{eq:ak}) is solved by
\[a_k(t) = i\cdot \int_0^t \sum_{l\in\zset^n} \langle a_{k-l}(t),k\rangle \cdot \pset_k a_l(t)\cdot e^{\nu\cdot k^2\cdot s}~\diff s\cdot e^{-\nu\cdot k^2\cdot t} + a_k(0)\cdot e^{-\nu\cdot k^2\cdot t} \tag{$+$}\]
which implies the bound
\begin{align*}
\|u(x,t)\|_{\sA,1} &= \sum_{k\in\zset^n} |k|\cdot |a_k(t)|\\
&\overset{(+)}{\leq} \sum_{k\in\zset^n} |k|\cdot \left| \int_0^t \sum_{l\in\zset^n} \langle a_{k-l}(t),k\rangle \cdot \pset_k a_l(t)\cdot e^{\nu\cdot k^2\cdot s}~\diff s\cdot e^{-\nu\cdot k^2\cdot t} \right|\\
&\quad + \sum_{k\in\zset^n} |k|\cdot |a_k(0)|\cdot e^{-\nu\cdot k^2\cdot t}\\
&\leq \sum_{k\in\zset^n} |k|\cdot \int_0^t \sum_{l\in\zset^n\setminus\{0\}} |a_{k-l}(s)|\cdot |k|\cdot |a_l(s)|\cdot e^{\nu\cdot k^2\cdot s}~\diff s\cdot e^{-\nu\cdot k^2\cdot t}\\
&\quad + \sum_{k\in\zset^n} |k|\cdot |a_k(0)|\cdot e^{-\nu\cdot k^2\cdot t} \tag{\#}\\
&\leq \sum_{l,k\in\zset^n\setminus\{0\}}\!\!\!\!\! \nu^{-1}\cdot A_{k-l}\cdot A_l\cdot \left[ 1 - e^{-\nu\cdot k^2\cdot t}\right] + \sum_{k\in\zset^n} |k|\cdot |a_k(0)|\cdot e^{-\nu\cdot k^2\cdot t}\\
&\leq \nu^{-1}\cdot \left( \sum_{k\in\zset^n\setminus\{0\}} A_k\right)^2 + \|u_0\|_{\sA,1}\cdot e^{-\nu\cdot t}\\
&\leq \nu^{-1}\cdot \|u_0(x)\|_{\sA,0}^2\cdot \exp\left(\frac{1}{2\nu} \int_0^{t} \|u(x,s)\|_{\sA,0}^2~\diff s\right) + \|u_0\|_{\sA,1}\cdot e^{-\nu\cdot t}.
\end{align*}
Therefore, as long as $\|u(x,t)\|_{\sA,0}$ is bounded, so is $\|u(x,t)\|_{\sA,1}$. If for all $t\in [0,T]$ we have the a priori bound $\|u(x,t)\|_{\sA,0} \leq C < \infty$, then $\|u(x,t)\|_{\sA,1}$ is bounded for all $t\in [0,T]$ and by \Cref{lem:existence} the smooth solution exists for all $t\in [0,T]$.

To prove the bounds of $T^*$ in (\ref{eq:timeStar}) we now only have to bound $\|u(x,t)\|_{\sA,0}$. But as in the proof of \Cref{lem:existence} we have for $\|u_0\|_{\sA,0} + \delta\leq \nu$ with $\delta\geq 0$ the a priori bound $\|u(x,t)\|_{\sA,0} \leq \|u_0\|_{\sA,0}\cdot e^{-\delta\cdot t}$ for all $t\in [0,\infty)$, i.e., $T^* = \infty$. And if $\|u_0\|_{\sA,0}>\nu$ then we only have (\ref{eq:normBound}) with $v(x,t) = -u(x,t)$, i.e.,
\[\|u(x,t)\|_{\sA,0} \leq \|u_0\|_{\sA,0}\cdot \exp\left(\frac{1}{4\nu}\int_0^t \|u(x,s)\|_{\sA,0}^2~\diff s\right),\]
which ensured in (\ref{eq:f0}) that $\|u(x,t)\|_{\sA,0}$ stays finite at least for $T^* \geq 2\nu\cdot \|u_0\|_{\sA,0}^{-2}$.

Uniqueness follows from the Banach fixed point theorem since we only need to control $\|u(\,\cdot\,,t)\|_{\sA,0}$ resp.\ $\|u(\,\cdot\,,t)\|_\sA = \|u(\,\cdot\,,t)\|_{\sA,0} + |a_0(0)|$ and can therefore work on a space similar to $l^1$. Let $L := \frac{1}{5\cdot \|u_0\|_{\sA}}$. There exists a $T\in (0,T^*)$ such that
\[\frac{1}{\nu\cdot |k|}\left( 1 - e^{-\nu k^2\cdot t}\right) \leq L\]
for all $k\in\nset_0^n\setminus\{0\}$ and $t\in [0,T]$ and
\[M := \left\{ \sum_{k\in\nset_0^n} f_k(t)\cdot e^{i\cdot k\cdot x} \,\middle|\, f_k\in C([0,T],\cset^n),\ \sum_{k\in\nset_0^n} \max_{t\in [0,T]} |f_k(t)| \leq 2\|u_0\|_{\sA} \right\}\]
is a closed bounded convex set of a Banach space with the norm
\[\|f\|_M := \sum_{k\in\nset_0^n} \max_{t\in [0,T]} |f_k(t)|\]
for $f(x,t) = \sum_{k\in\nset_0^n} f_k(t)\cdot e^{i\cdot k\cdot x}$. With $u_0(x) = \sum_{k\in\nset_0^n} a_k(0)\cdot e^{i\cdot k\cdot x}$ we define
\[(S f)(x,t) := \sum_{k,l\in\nset_0^n} \int_0^t e^{\nu\cdot |k|^2\cdot (s-t)}\cdot \langle f_{k-l}(s),k\rangle\cdot \pset_k f_l(s)~\diff s\cdot e^{i\cdot k\cdot x} + a_k(0)\cdot e^{-\nu\cdot k^2\cdot t},\]
i.e., by (\&) $u$ is in $M$ and it is a fixed point of $S$. Let $f\in M$, then
\begin{align*}
\|S f\|_M &\leq  \left\| \sum_{k,l\in\nset_0} \int_0^t e^{\nu\cdot k^2\cdot (s-t)}\cdot \langle f_{l-k}(s),k\rangle\cdot\pset_k f_l(s)~\diff s\cdot e^{i\cdot k\cdot x}\right\|_M + \|u_0\|_\sA\\
&\leq \sum_{k,l\in\nset_0^n, k\neq 0} \frac{1}{\nu\cdot |k|}\cdot \left( 1 - e^{-\nu\cdot |k|^2\cdot t}\right)\cdot \max_{t\in [0,T]} |f_{k-l}(t)|\cdot \max_{t\in [0,T]} |f_l(t)| + \|u_0\|_\sA\\
&\leq \frac{4\|f\|_M^2}{5\cdot \|u_0\|_\sA} + \|u_0\|_\sA\\
&\leq 2\|u_0\|_\sA,
\end{align*}
i.e., $Sf\in M$ and therefore $S:M\to M$. We now show that $S: M\to M$ is a contraction. Let $f,g\in M$ with Fourier coefficients $f_k$ and $g_k$, then
\begin{align*}
\|Sf - Sg \|_M &= \bigg\| \sum_{k,l\in\nset_0^n} \int_0^t e^{\nu\cdot k^2\cdot (s-t)}\cdot [\langle f_{k-l}(s),k\rangle\cdot \pset_k f_l(s)\\
&\qquad\qquad\qquad\qquad\qquad - \langle g_{k-l}(s),k\rangle\cdot \pset_k g_l(s)] ~\diff s\cdot e^{i\cdot k\cdot x} \bigg\|_M\\
&= \bigg\| \sum_{k,l\in\nset_0^n} \int_0^t e^{\nu\cdot k^2\cdot (s-t)}\cdot [\langle f_{k-l}(s)-g_{k-l}(s),k\rangle\cdot \pset_k f_l(s)\\
&\qquad\qquad\qquad\qquad\qquad - \langle g_{k-l}(s),k\rangle\cdot \pset_k (g_l(s)-f_l(s))] ~\diff s\cdot e^{i\cdot k\cdot x} \bigg\|_M\\
&\leq L\cdot (\|f\|_M + \|g\|_M)\cdot \|f-g\|_M\\
&\leq \frac{4}{5}\cdot \|f-g\|_M.
\end{align*}
Hence, $S:M\to M$ is a contraction on the closed, bounded, and convex subset $M$ of a Banach space and the Banach fixed point theorem shows that $S$ has a unique fixed point in $M$. Since $u\in M$ is a fixed point, it is therefore unique for all $t\in [0,T]$. By induction we find that $u$ is unique on any $[0,T]\subset [0,T^*)$, i.e., it is unique for all $t\in [0,T^*)$.
\end{proof}

Note, that the construction of a solution $u$ of the Navier--Stokes equation with its bounds on the semi-norms $\|\cdot\|_{\sA,d}$ is proved by a splitting algorithm in \Cref{thm:main} and \Cref{lem:existence}. We found in \Cref{thm:mainNS} that its existence is controlled by $\|u(\,\cdot\,,t)\|_{\sA,0}$. Hence, in the uniqueness proof of $u$ it is sufficient to control only $\|u(\,\cdot\,,t)\|_\sA = \|u(\,\cdot\,,t)\|_{\sA,0} + |a_0(0)|$, i.e., we can work in a $l^1$-like Banach space and can use the Banach fixed point theorem. We do not need to control any derivatives and do not need to use fixed point theorems on Montel (or Fr\'echet) spaces.

That existence and uniqueness is solely controlled by $\|u(\,\cdot\,,t)\|_{\sA,0}$ provides us with a similar finite break down criteria ($T^*$ finite) as by Hiroshi Fujita and Tosio Kato \cite{fujita64}
\[\sup_{t\in (0,T^*)} \|u(\,\cdot\,,t)\|_{H^1} = \infty.\]
or by J.\ Thomas Beale, Tosio Kato, and Andrew J.\ Majda \cite{beale84}:
\[\lim_{t\nearrow T^*} \|\rot u(\,\cdot\,,t)\|_\infty = \infty.\]

\begin{cor}
Let $n\in\nset$, $n\geq 2$, $T^*\in (0,\infty)$, and $u_0\in C^\infty(\tset^n)$ with $\divv u_0 =0$. Then
\begin{enumerate}[(i)]
\item the unique smooth solution $u$ of the $n$-dimensional periodic Navier--Stokes equation exists for all $t\in [0,T^*)$ but can not be extended beyond $T^*$
\end{enumerate}
if and only if
\begin{enumerate}[(i)]\setcounter{enumi}{1}
\item $\lim_{t\nearrow T^*} \|u(\,\cdot\,,t)\|_{\sA,0} = \infty$.
\end{enumerate}
\end{cor}

Since $a_0(t)$ is constant, (ii) is of course equivalent to $\lim_{t\nearrow T} \|u(\,\cdot\,,t)\|_\sA = \infty$. More known conditions for a finite break down can be found e.g.\ in \cite{wahl85} ,\cite{temam01}, \cite{majda02}, \cite{bahour11}, \cite{lemari13}, and references therein.

\begin{rem}
For $n=2$ the Navier--Stokes problem is solved \cite{leray33}. In \cite[Lem.\ 9.1.6]{kreiss89} it is nicely shown that $\|\rot u(\,\cdot\,,t)\|_{H^1} \leq C$ holds for all $t\geq 0$. With $u(x,t) = (u_1(x_1,x_2,t),u_2(x_1,x_2,t),0)^t$ and $\rot u(x,t) = (0,0,\omega_3(x_1,x_2,t))^t$ we get from \Cref{lem:seminormProperties}(vi)
\begin{align*}
\|u(\,\cdot\,,t)\|_{\sA,0} &\leq \underbrace{\sqrt{\sum_{k\in\zset^2\times\{0\},k\neq 0} |k|^{-4}}}_{<\infty} \cdot \sqrt{\sum_{k\in\zset^2\times\{0\},k\neq 0} |k|^{4}\cdot |u_k|^2 }\\
&\leq \sqrt{\sum_{k\in\zset^2\times\{0\},k\neq 0} |k|^{-4}} \cdot \sqrt{\sum_{k\in\zset^2\times\{0\},k\neq 0} |k|^{2}\cdot |k\times u_k|^2 }\\
&\leq K\cdot \|\rot u(\,\cdot\,,t)\|_{H^1}\\
&\leq K\cdot C < \infty
\end{align*}
for all $t\geq 0$. Hence, the a priori bound in \Cref{thm:mainNS} is fulfilled for all $t\geq 0$ and the $2$-dimensional periodic Naiver--Stokes equation has a unique smooth solution $u$ for all $t\geq 0$. In summary, \Cref{lem:seminormProperties}(vi) allows to apply $H^s$-bounds to $\|\cdot\|_{\sA,d}$.\exmsymbol
\end{rem}

\section{Examples for the growth of the Fourier coefficients}
\label{sec:examples}

We have seen in our main result \Cref{thm:mainNS} that the $l^1$-norm $\|u(x,t)\|_{\sA,0}$ of the Fourier coefficients $a_k(t)$ of a solution $u$ of the Naiver--Stokes equation solely determines existence, smoothness, and uniqueness. We therefore want to understand the time development of the Fourier coefficients in more detail. We present here three simple but explicit examples of solutions of (\ref{eq:ens2}).

The first two examples (\Cref{exm:simpleBreakDown} and \ref{exm:noBreakDown}) can be seen as linear versions of Burgers' equation with and without viscosity. The third example (\Cref{exm:besselConnection}) is without viscosity but for dimensions $n\geq 3$. It will reveal the connection of the time dependent Fourier coefficients $a_k(t)$ to the Bessel functions.

For the first two examples let $\nu\geq 0$ and let us look at the initial value problem
\begin{equation}\label{eq:initialOneDim}
\begin{split}
\begin{pmatrix} \partial_t f_1(x,t)\\ \partial_t f_2(x,t)\end{pmatrix} &= \nu \begin{pmatrix}
\partial_x^2 f_1(x,t)\\ \partial_x^2 f_2(x,t)\end{pmatrix} + \begin{pmatrix}\cos x & -\sin x\\ \sin x & \cos x\end{pmatrix} \cdot \begin{pmatrix} \partial_x f_1(x,t) \\ \partial_x f_2(x,t)\end{pmatrix}\\
\begin{pmatrix} f_1(x,0) \\ f_2(x,0)\end{pmatrix} &= \begin{pmatrix} \cos x\\ \sin x\end{pmatrix}
\end{split}
\end{equation}
on $\tset\times [0,T^*)$. By the Euler formula $e^{i\cdot x} = \cos x + i\cdot\sin x$ the initial value problem (\ref{eq:initialOneDim}) is equivalent to the complex-valued ($f = f_1 + i\cdot f_2$) initial value problem
\begin{equation}\label{eq:initialOneDimComplex}
\begin{split}
\partial_t f(x,t) &= \nu\cdot\partial_x^2 f(x,t) + e^{i\cdot x}\cdot\partial_x f(x,t)\\
f(x,0) &= e^{i\cdot x}
\end{split}
\end{equation}
on $\tset\times [0,T^*)$, $T^* > 0$. \Cref{exm:simpleBreakDown} shows for $\nu = 0$ that (\ref{eq:initialOneDim}) and (\ref{eq:initialOneDimComplex}) have a unique smooth solution which breaks down in finite time at $T^* = 1$.

\begin{exm}\label{exm:simpleBreakDown}
The initial value problem (\ref{eq:initialOneDimComplex}) with $\nu = 0$ has the unique solution
\[f(x,t) = \sum_{k\in\nset} a_k(t)\cdot e^{i\cdot k\cdot x} \qquad\text{with}\qquad a_k(t) = i^{k-1}\cdot t^{k-1},\]
i.e., $f\in C^\infty([0,1),C^\infty(\tset,\cset))$ but it breaks down at $T^*=1$ since $\|f(x,t)\|_\sA = \frac{1}{1-t}$.
\end{exm}
\begin{proof}
With the Ansatz
\begin{equation}\label{eq:ansatz}
f(x,t) = \sum_{k\in\nset} a_{k}(t)\cdot e^{i\cdot k\cdot x}
\end{equation}
for the solution we get
\[\sum_{k\in\nset} \dot{a}_{k}(t)\cdot e^{i\cdot k\cdot x} = \partial_t f(x,t) = e^{i\cdot x}\cdot\partial_x f(x,t) = \sum_{k\in\nset} a_{k}(t)\cdot i\cdot k\cdot e^{i\cdot (k+1)\cdot x}\]
i.e., the time-dependent coefficients $a_{k}(t)$ fulfill for $k=1$
\begin{align*}
\dot{a}_{1}(t) = 0\quad\text{with}\quad a_1(0) = 1 \qquad&\Rightarrow\qquad a_1(t) = 1,
\intertext{and for $k\geq 2$}
\dot{a}_{k}(t)=i\cdot (k-1)\cdot a_{k-1}(t)\quad\text{with}\quad a_{k}(0)=0\qquad&\Rightarrow\qquad a_k(t)=i^{k-1}\cdot t^{k-1}.
\end{align*}
Hence, we have the solution
\[f(x,t) = \sum_{k\in\nset} i^{k-1}\cdot t^{k-1}\cdot e^{i\cdot k\cdot x} \qquad\text{with}\qquad \|\partial_x^l f(x,t)\|_\sA = \sum_{k\in\nset} k^l\cdot t^{k-1}\]
and therefore $f(x,t)\in C^\infty([0,2\pi],\cset)$ for all $t\in [0,1)$ but breaks down at $T^* = 1$. In fact, we have $\|f(x,t)\|_\sA = \sum_{k=0}^\infty t^k = \frac{1}{1-t}$.
\end{proof}

So we see that (\ref{eq:initialOneDimComplex}) with $\nu=0$ breaks down in finite time $T^*=1$. But the next example shows that with $\nu> 0$ the solution of (\ref{eq:initialOneDimComplex}) never breaks down ($T^*=\infty$).

\begin{exm}\label{exm:noBreakDown}
The initial value problem (\ref{eq:initialOneDimComplex}) with $\nu > 0$ has the unique solution
\[f(x,t) = \sum_{k\in\nset} a_k(t)\cdot e^{i\cdot k\cdot x} \qquad\text{with}\qquad |a_k(t)| \leq \frac{2\cdot e^{-\nu\cdot t}}{\nu^{k-1}\cdot (k+1)!},\]
i.e., $f\in C^\infty([,\infty),C^\infty(\tset,\cset))$.
\end{exm}
\begin{proof}
With the Ansatz (\ref{eq:ansatz}) we get
\[\sum_{k\in\nset}\dot{a}_k(t)\cdot e^{i\cdot k\cdot x} = \sum_{k\in\nset}\left[-\nu\cdot k^2\cdot a_k(t)\cdot e^{i\cdot k\cdot x} +  i\cdot k\cdot a_k(t)\cdot e^{i\cdot (k+1)\cdot x}\right]\]
and therefore the time-dependent coefficients $a_k(t)$ fulfill for $k = 1$
\[\dot{a}_1(t) = -\nu\cdot a_1(t)\quad\text{with}\quad a_1(0)=1 \qquad\Rightarrow\qquad a_1(t) = e^{-\nu\cdot t}\]
and for $k\geq 2$ we have
\[\dot{a}_{k}(t) = -\nu\cdot k^2\cdot a_k(t) + i\cdot (k-1)\cdot a_{k-1}(t)\quad\text{with}\quad a_k(0)=0\]
which provides the induction
\begin{equation}\label{eq:akCalc}
a_k(t) = i\cdot (k-1)\cdot\int_0^t a_{k-1}(s)\cdot e^{\nu\cdot k^2\cdot s}~\diff s\cdot e^{-\nu\cdot k^2\cdot t}.
\end{equation}
Hence, we have $a_1(t) \geq 0$ and (\ref{eq:akCalc}) also implies $i^{-k+1}\cdot a_k(t) \geq 0$ for all $k\in\nset$. For $k\in\nset$ define
\[a_k^*(t) := \frac{2\cdot i^{k-1}}{\nu^{k-1}\cdot (k+1)!}\cdot e^{-\nu\cdot t}.\]
Then
\[a_1^*(t) = \frac{2\cdot i^{1-1}}{\nu^{1-1}\cdot (1+1)!}\cdot e^{-\nu\cdot t} = a_1(t),\]
i.e.,
\begin{equation}\label{eq:akBounds}
i^{-k+1}\cdot a_k^*(t) \geq i^{-k+1}\cdot a_k(t) \geq 0
\end{equation}
holds for $k=1$ and all $t\geq 0$. Assume (\ref{eq:akBounds}) holds for a $k\in\nset$ and all $t\geq 0$. Then
\begin{align*}
i^{-k}\cdot a_{k+1}(t) &= i^{-k}\cdot i\cdot k\cdot \int_0^t a_k(s)\cdot e^{\nu\cdot (k+1)^2\cdot s}~\diff s\cdot e^{-\nu\cdot (k+1)^2\cdot t}\\
&= k\cdot \int_0^t i^{-k+1}\cdot a_k(s)\cdot e^{\nu\cdot (k+1)^2\cdot s}~\diff s\cdot e^{-\nu\cdot (k+1)^2\cdot t}\\
&\leq k\cdot \int_0^t i^{-k+1}\cdot a_k^*(s)\cdot e^{\nu\cdot (k+1)^2\cdot s}~\diff s\cdot e^{-\nu\cdot (k+1)^2\cdot t}\\
&= k\cdot \int_0^t i^{-k+1}\cdot \frac{2\cdot i^{k-1}}{\nu^{k-1}\cdot (k+1)!}\cdot e^{-\nu\cdot s}\cdot e^{\nu\cdot (k+1)^2\cdot s}~\diff s\cdot e^{-\nu\cdot (k+1)^2\cdot t}\\
&= \frac{2\cdot k}{\nu^{k-1}\cdot (k+1)!}\cdot \int_0^t e^{\nu\cdot [(k+1)^2 - 1]\cdot s}~\diff s\cdot e^{-\nu\cdot (k+1)^2\cdot t}\\
&= \frac{2\cdot k}{(k+1)!\cdot \nu^k\cdot [(k+1)^2 - 1]}\cdot \left[ e^{\nu\cdot [(k+1)^2 - 1]\cdot s} \right]_{s=0}^t\cdot e^{-\nu\cdot (k+1)^2\cdot t}\\
&= \frac{2}{(k+2)!\cdot \nu^k}\cdot \left[ e^{\nu\cdot [(k+1)^2 - 1]\cdot t} - 1 \right] \cdot e^{-\nu\cdot (k+1)^2\cdot t}\\
&= \frac{2}{(k+2)!\cdot \nu^k}\cdot \left[ e^{-\nu\cdot t} - e^{-\nu\cdot (k+1)^2\cdot t} \right]\\
&\leq \frac{2}{(k+2)!\cdot \nu^k}\cdot e^{-\nu\cdot t}\\
&= i^{-k}\cdot a_{k+1}^*(t),
\end{align*}
i.e., (\ref{eq:akBounds}) holds also for $k+1$. Therefore, (\ref{eq:akBounds}) holds for all $k\in\nset$ and all $t\geq 0$. Hence,
\[|a_k(t)| \leq |a_k^*(t)| = \frac{2\cdot e^{-\nu\cdot t}}{\nu^{k-1}\cdot (k+1)!}\]
and
\[\|\partial_x^l f(\,\cdot\,,t)\|_\sA = \sum_{k\in\nset} |k^l\cdot a_k(t)| \leq \sum_{k\in\nset} |k^l\cdot a_k^*(t)| = 2\cdot e^{-\nu\cdot t}\cdot \sum_{k\in\nset} \frac{k^l\cdot (\nu^{-1})^{k-1}}{(k+1)!} < \infty\]
hold for all $k\in\nset$ and $t\geq 0$. In summary, for every $l\in\nset_0$ there exists a $c_l > 0$ dependent on $\nu>0$ but independent on $t$ such that
\[\| \partial_x^l f(\,\cdot\,,t)\|_\infty \leq \|\partial_x^l f(\,\cdot\,,t)\|_\sA \leq c_l\cdot e^{-\nu\cdot t}\]
and $f(\,\cdot\,,t)\in C^\infty ([0,2\pi],\cset)$ for all $t\geq 0$. For $l=0$ we have the explicit bound
\[\|f(\,\cdot\,,t)\|_\infty \leq \|f(\,\cdot\,,t)\|_\sA \leq 2\cdot e^{-\nu\cdot t}\cdot [\nu^2\cdot \exp(\nu^{-1}) - \nu^2 - \nu]\]
for all $t\geq 0$. All $\|\partial_x^l f(\,\cdot\,,t)\|_\infty$ decay exponentially in time with rate $-\nu$.
\end{proof}

That $\nu\Delta$ for any $\nu>0$ removes the finite break down is also known for Burgers' equation. Burgers' equation with viscosity is by the Cole--Hopf transformation \cite{hopf50,cole51} equivalent to the heat equation, see e.g.\ \cite[p.\ 207]{evans10} or \cite[Sec.\ 1.1.5]{handbookNonLinPDEs}. Therefore Burgers' equation with viscosity has a solution for all $t\in [0,\infty)$. For a proof on $\tset$ without a Cole--Hopf transformation see e.g.\ \cite[Ch.\ 4]{kreiss89}.

In the third example we will now see the connection of the time dependent Fourier coefficients $a_k(t)$ to the Bessel functions.

In \Cref{thm:main} the coefficients
\[\pset_k a_l(t_i)\]
appeared and since all $\pset_k:\cset^n\to\cset^n$ are projections we used the bounds
\[|\pset_k a_l(t_i)| \leq |a_l(t_i)|,\]
of course in the $l^2$-norm $|\,\cdot\,|$ on $\cset^n$. But if we have $a_l(t_i)\perp k$ we get the apparently worst case
\[\pset_k a_l(t_i) = a_l(t_i).\]
So we will look in \Cref{exm:besselConnection} at this special case which can only appear for $n\geq 3$. We find a solution for
\[v(x,t)=b_l(t)\cdot e^{i\cdot l\cdot x}+b_{-l}(t)\cdot e^{-i\cdot l\cdot x}\in C^\infty([0,\infty),C^\infty(\tset^n))^n.\]
Since $v(x,t)$ shall be real we always have $\overline{b_l(t)} = b_{-l}(t)$. \Cref{exm:besselConnection} reveals the connection to the well-studied Bessel functions of the first kind of order $j\in\nset_0$ \cite{watsonBesselFunctions}:
\[J_j(t) := \sum_{a\in\nset_0} \frac{(-1)^{a}\cdot t^{2a+j}}{a!\cdot (a+j)!\cdot 2^{2a+j}}.\]

\begin{exm}\label{exm:besselConnection}
Let $n\in\nset$, $n\geq 3$, $k,l\in\zset^n$ with $l\neq 0$, $a_k(0)\in\cset^n$ with $a_k(0)\perp l$, and $b_l,b_{-l}\in C([0,\infty),\cset)^n$ with $b_l(t)\cdot l = b_{-l}(t)\cdot l = 0$ and $\overline{b_l(t)} = b_{-l}(t)$ for all $t\in [0,\infty)$. Then the initial value problem
\begin{align*}
\partial_t u(x,t) &= \pset\big[(b_l(t)\cdot e^{i\cdot l\cdot x} + b_{-l}(t)\cdot e^{-i\cdot l\cdot x})\nabla u(x,t)\big]\\
u(x,0) &= a_k(0)\cdot e^{i k\cdot x}
\intertext{has with $B_+(t) := i\cdot \int_0^t \langle b_l(s),k\rangle~\diff s \neq 0$ the unique $C^\infty$-solution}
u(x,t) &= \sum_{j\in\zset} a_k(0)\cdot \left( \frac{B_+(t)}{|B_+(t)|} \right)^j\cdot J_{|j|}(2\cdot |B_+(t)|) \cdot e^{i(k+jl)\cdot x}
\end{align*}
where $J_j$ is the Bessel function of the first kind of order $j$. For $B_+(t) = 0$ the unique solution is $u(x,t) = a_k(0)\cdot e^{i\cdot k\cdot x}$.
\end{exm}
\begin{proof}
Set $\beta_+(t) := i\cdot \langle b_{l}(t),k\rangle$ and $\beta_-(t) := i\cdot \langle b_{-l}(t),k\rangle$. Since $a_k(0)\perp l$ we have
\[\pset_{k+jl} a_k(0) = a_k(0)\]
for all $j\in\zset$. The initial value problem is equivalent to 
\[\dot{a}_{k+jl}(t) = \beta_+(t)\cdot a_{k+(j-1)l}(t) + \beta_-(t)\cdot a_{k+(j+1)l}(t). \tag{$*$}\]
Therefore, $a_{k'}(t) = 0$ for all $k'\neq k+jl$ for all $j\in\zset$. Defining the shifts $S_{+}$ and $S_{-}$ by
\[S_{+}(a_{k'})_{k'\in\zset^n} := (a_{k'+l})_{k'\in\zset} \quad\text{and}\quad S_{-}(a_{k'})_{k'\in\zset^n} := (a_{k'-l})_{k'\in\zset},\]
i.e.,
\[S_{+} S_{-} = S_{-} S_{+} = \id,\tag{$**$}\]
we can write ($*$) as
\[(\dot{a}_{k+jl}(t))_{j\in\zset} = (\beta_+(t)\cdot S_{+} + \beta_-(t)\cdot S_{-})\cdot (a_{k+jl}(t))_{j\in\zset}\]
which has the (unique) solution
\[(a_{k+jl}(t))_{j\in\zset} = \exp\left( \int_0^t \beta_+(s)~\diff s\cdot S_{+} + \int_0^t \beta_-(s)~\diff s\cdot S_{-}\right)\cdot (a_{k+jl}(0))_{j\in\zset}\]
with of course $a_{k+jl}(0) = a_k(0)$ for $j=0$ and $a_{k+jl}(0)=0$ for all $j\neq 0$. Check by differentiation and use that $S_+$ and $S_-$ commute by ($**$).

If $\int_0^t \beta_+(s)~\diff s = 0$, then $u(x,t) = a_k(0)\cdot e^{i\cdot k\cdot x}$. So assume $\int_0^t \beta
_+(s)~\diff s\neq 0$.

Let $e_{k+jl}$ be the vector with the unit matrix $\id\in\cset^{n\times n}$ at position $k+jl$ and the zero matrix $0\in\cset^{n\times n}$ everywhere else. Then
\[a_{k+jl}(t) = \langle e_{k+jl},\exp\left( \int_0^t \beta_+(s)~\diff s\cdot S_{+} + \int_0^t \beta_-(s)~\diff s\cdot S_{-}\right) e_k\rangle\cdot a_k(0)\]
for all $j\in\zset$. Hence, it is sufficient to calculate
\[\langle e_{k+jl},\exp\left( \int_0^t \beta_+(s)~\diff s\cdot S_{+} + \int_0^t \beta_-(s)~\diff s\cdot S_{-}\right) e_k\rangle \in\cset^{n\times n}.\]
Let $j = 2\alpha\in\nset_0$ be non-negative and even. For simplicity we use
\[B_+(t) := \int_0^t \beta_+(s)~\diff s = i\cdot \int_0^t \langle b_l(s),k\rangle~\diff s\]
and
\[B_-(t) := \int_0^t \beta_-(s)~\diff s = i\cdot \int_0^t \langle b_{-l}(s),k\rangle~\diff s.\]
Then since $\overline{b_l(t)} = b_{-l}(t)$ we have
\[\overline{B_+(t)} = - B_-(t) \qquad\text{and}\qquad B_+(t)\cdot B_-(t) = - |B_+(t)|^2 = - |B_-(t)|^2.\]

Since $S_+$ and $S_-$ commute by ($**$) we have
\begin{align*}
&\quad\ \langle e_{k+2\alpha l},\exp(B_+(t)\cdot S_{+} + B_-(t)\cdot S_{-}) e_k\rangle\\
&= \langle S_{+}^{2\alpha} e_k, \sum_{a=0}^\infty \frac{1}{a!} (B_+(t)\cdot S_{+} + B_-(t)\cdot S_{-})^a e_k\rangle\\
&= \langle e_k, S_{-}^{2\alpha} \sum_{a=0}^\infty \frac{1}{a!} \sum_{b=0}^a \begin{pmatrix} a\\ b\end{pmatrix} B_+(t)^b\cdot B_-(t)^{a-b}\cdot S_{+}^b\cdot S_{-}^{a-b} e_k\rangle\\
&= \langle e_k, \sum_{a=0}^\infty \frac{1}{a!}\sum_{b=0}^a \begin{pmatrix}a\\ b\end{pmatrix}\cdot B_+(t)^b\cdot B_-(t)^{a-b}\cdot S_{-}^{a-2b+2\alpha} e_k\rangle
\intertext{since $a-2b+2\alpha$ must be zero by the orthogonality of the $e_k$'s, $a$ must be even and so we sum over $2a$}
&= \langle e_k, \sum_{a=0}^\infty \frac{1}{(2a)!} \sum_{b=0}^{2a} \begin{pmatrix}2a\\ b\end{pmatrix}\cdot B_+(t)^b\cdot B_-(t)^{2a-b}\cdot S_{-}^{2a-2b+2\alpha} e_k\rangle
\intertext{here again the orthogonality of the $e_k$'s implies $2a-2b+2\alpha=0$, i.e., $b = a+\alpha$,}
&= \sum_{a=0}^\infty \frac{1}{(2a)!}\cdot\begin{pmatrix} 2a\\ a+\alpha\end{pmatrix}\cdot B_+(t)^{a+\alpha}\cdot B_-(t)^{a-\alpha}
\intertext{so the summation over $a$ begins at $a=\alpha$, since otherwise $\left(\begin{smallmatrix} 2a\\ a+\alpha\end{smallmatrix}\right)$ is zero,}
&= \sum_{a=\alpha}^\infty \frac{1}{(2a)!}\cdot\begin{pmatrix} 2a\\ a+\alpha\end{pmatrix}\cdot B_+(t)^{a+\alpha}\cdot B_-(t)^{a-\alpha}\\
&= \sum_{a=\alpha}^\infty \frac{1}{(a-\alpha)!\cdot (a+\alpha)!}\cdot B_+(t)^{a+\alpha}\cdot B_-(t)^{a-\alpha}
\intertext{and shifting the summation back to start at $a=0$}
&= \sum_{a=0}^\infty \frac{1}{a!\cdot (a+2\alpha)!}\cdot B_+(t)^{a+2\alpha}\cdot B_-(t)^{a}\\
&= B_+(t)^{2\alpha}\cdot \sum_{a=0}^\infty \frac{(-1)^a\cdot |B_+(t)|^{2a}}{a!\cdot (a+2\alpha)!}\\
&= \left(\frac{B_+(t)}{|B_+(t)|}\right)^{2\alpha}\cdot \sum_{a=0}^\infty \frac{(-1)^a\cdot (2\cdot |B_+(t)|)^{2a+2\alpha}}{a!\cdot (a+2\alpha)!\cdot 2^{2a+2\alpha}}\\
&= \left(\frac{B_+(t)}{|B_+(t)|}\right)^{2\alpha}\cdot J_{2\alpha}(2\cdot |B_+(t)|).
\end{align*}
In the same way we calculate the cases $k+(2\alpha+1)l$, $k-2\alpha l$, and $k-(2\alpha+1)l$ for all $\alpha\in\nset_0$, see \Cref{app:bessel}, and we therefore get
\[a_{k+jl}(t) = a_k(0)\cdot \left(\frac{B_+(t)}{|B_+(t)|}\right)^j\cdot J_{|j|}(2\cdot |B_+(t)|)\]
for all $j\in\zset$.
\end{proof}

$J_0(x)^2 + 2\sum_{j=1}^\infty J_j(x)^2 = 1$ \cite[§2.5 (3)]{watsonBesselFunctions} gives $\|u(x,t)\|_{L^2(\tset^n)} = \|u_0\|_{L^2(\tset^n)}$. $\frac{1}{2}z^2 = \sum_{j\in\nset} \varepsilon_j j^2 J_j(z)^2$ with $\varepsilon_j = \big\{\!\begin{smallmatrix} 1 & \text{for}\ j=1,\\ 2 & \text{for}\ j\geq 2\phantom{,}\end{smallmatrix}$ \cite[§2.72 (3)]{watsonBesselFunctions} for $k_1=0$ implies
\[\|\partial_1 u(x,t)\|_{L^2(\tset^n)} = (2\pi)^n |a_k(0)\cdot l_1|\cdot \sqrt{\frac{z^2}{2} + J_1(z)^2} \leq (2\pi)^n \sqrt{3}\cdot |a_k(0)\cdot l_1\cdot B_+(t)|\]
with $z = 2\cdot |B_+(t)|$ and for $k_1\in\zset$ gives
\[\|\partial_1 u(x,t)\|_{L^2(\tset^n)} \leq \|\partial_1 u_0\|_{L^2(\tset^n)} + (2\pi)^n \sqrt{3}\cdot |a_k(0)\cdot l_1\cdot B_+(t)|.\]
We therefore get the lower bound
\[\|\partial_1 u(x,t)\|_{L^2(\tset^n)} \geq \sqrt{2}\cdot (2\pi)^n\cdot |a_k(0)\cdot l_1|\cdot \left|\int_0^t \langle b_l(s),k\rangle~\diff s\right|\]
on the $L^2$-norm of the first derivative.


\providecommand{\bysame}{\leavevmode\hbox to3em{\hrulefill}\thinspace}
\providecommand{\MR}{\relax\ifhmode\unskip\space\fi MR }
\providecommand{\MRhref}[2]{%
  \href{http://www.ams.org/mathscinet-getitem?mr=#1}{#2}
}
\providecommand{\href}[2]{#2}

\appendix 
\section{Bessel function calculations of \Cref{exm:besselConnection}}\label{app:bessel}

\begin{proof}[Inserted Proof]
We retain the notations in the proof of \Cref{exm:besselConnection}. We calculate
\[\langle e_{k+jl},\exp\left( B_+(t)\cdot S_{+} + B_-(t)\cdot S_-\right) e_k\rangle \in\cset^{n\times n} \tag{\#}\]
for the remaining cases $j = 2\alpha + 1$, $-2\alpha$, and $-2\alpha-1$ with $\alpha\in\nset_0$. We assume $B_+(t)\neq 0$ for all three cases, otherwise (\#) is zero for each case.

Let $j = 2\alpha + 1$ with $\alpha\in\nset_0$. Then
\begin{align*}
&\phantom{=}\;\;\langle e_{k+(2\alpha+1)l},\exp\left( B_+(t)\cdot S_{+} + B_-(t)\cdot S_-\right) e_k\rangle\\
&= \langle S_+^{2\alpha+1} e_k, \sum_{a=0}^\infty \frac{1}{a!} (B_+(t)\cdot S_+ + B_-(t)\cdot S_-)^a e_k\rangle\\
&= \langle e_k, S_-^{2\alpha+1}\cdot \sum_{a=0}^\infty \frac{1}{a!}\cdot\sum_{b=0} \begin{pmatrix} a\\ b\end{pmatrix}\cdot B_+(t)^b\cdot B_-(t)^{a-b}\cdot S_+^b \cdot S_-^{a-b} e_k\rangle\\
&= \langle e_k, \sum_{a=0}^\infty \frac{1}{a!}\cdot \sum_{b=0}^a \begin{pmatrix} a\\ b\end{pmatrix}\cdot B_+(t)^b\cdot B_-(t)^{a-b}\cdot S_-^{a-2b+2\alpha+1} e_k\rangle
\intertext{here $a-2b+2\alpha+1$ must be zero, i.e., we sum only over the odd $a$'s,}
&= \langle e_k \sum_{a=0}^\infty \frac{1}{(2a+1)!}\cdot \sum_{b=0}^{2a+1} \begin{pmatrix} 2a+1\\b\end{pmatrix}\cdot B_+(t)^b\cdot B_-(t)^{2a+1-b}\cdot S_-^{2a-2b+2\alpha+2} e_k\rangle
\intertext{hence, $2a-2b+2\alpha+2$ must be zero, i.e., $b = a+\alpha+1$,}
&= \sum_{a=0}^\infty \frac{1}{(2a+1)!}\cdot \begin{pmatrix}2a+1\\ a+\alpha+1\end{pmatrix}\cdot B_+(t)^{a+\alpha+1}\cdot B_-(t)^{a-\alpha}
\intertext{the summation can start at $a=\alpha$, otherwise $\left(\begin{smallmatrix} 2a+1\\ a+\alpha+1\end{smallmatrix}\right)$ is zero,}
&= \sum_{a=\alpha}^\infty \frac{1}{(2a+1)!}\cdot \begin{pmatrix}2a+1\\ a+\alpha+1\end{pmatrix}\cdot B_+(t)^{a+\alpha+1}\cdot B_-(t)^{a-\alpha}\\
\intertext{shifting $a$ back to start at $a=0$ gives}
&= \sum_{a=0}^\infty \frac{1}{(2a+2\alpha+1)!}\cdot \begin{pmatrix}2a+2\alpha+1\\ a+2\alpha+1\end{pmatrix}\cdot B_+(t)^{a+2\alpha+1}\cdot B_-(t)^{a} \\
&= \sum_{a=0}^\infty \frac{B_+(t)^{a+2\alpha+1}\cdot B_-(t)^{a}}{(a+2\alpha+1)!\cdot a!}\\
&= B_+(t)^{2\alpha+1}\cdot\sum_{a=0}^\infty \frac{(-1)^a \cdot |B_+(t)|^{2a}}{(a+2\alpha+1)!\cdot a!}\\
&= \left(\frac{B_+(t)}{|B_+(t)|}\right)^{2\alpha+1}\cdot\sum_{a=0}^\infty \frac{(-1)^a\cdot (2\cdot|B_+(t)|)^{2a+2\alpha+1}}{(a+2\alpha+1)!\cdot a!\cdot 2^{2a+2\alpha+1}}\\
&= \left(\frac{B_+(t)}{|B_+(t)|}\right)^{2\alpha+1}\cdot J_{2\alpha+1}(2\cdot |B_+(t)|).
\end{align*}
Hence the statement is also true for $j=2\alpha+1$ with $\alpha\in\nset_0$.

Now, let $j=-2\alpha$ with $\alpha\in\nset$. Then
\begin{align*}
&\phantom{=}\;\;\langle e_{k-2\alpha l},\exp\left( B_+(t)\cdot S_{+} + B_-(t)\cdot S_-\right) e_k\rangle\\
&= \langle S_-^{2\alpha} e_k, \sum_{a=0}^\infty \frac{1}{a!} (B_+(t)\cdot S_+ + B_-(t)\cdot S_-)^a e_k\rangle\\
&= \langle e_k, S_+^{2\alpha}\cdot \sum_{a=0}^\infty \frac{1}{a!}\cdot \sum_{b=0}^a \begin{pmatrix} a\\ b\end{pmatrix}\cdot B_+(t)^{a-b}\cdot B_-(t)^b\cdot S_+^{a-b}\cdot S_-^b e_k\rangle\\
&= \langle e_k, \sum_{a=0}^\infty \frac{1}{a!}\cdot \sum_{b=0}^a \begin{pmatrix} a\\ b\end{pmatrix}\cdot B_+(t)^{a-b}\cdot B_-(t)^{b}\cdot S_+^{a-2b+2\alpha} e_k\rangle
\intertext{hence $a-2b+2\alpha$ must be zero, i.e., the summation is over all even $a$'s,}
&= \langle e_k, \sum_{a=0}^\infty \frac{1}{(2a)!}\cdot \sum_{b=0}^{2a} \begin{pmatrix} 2a\\ b\end{pmatrix}\cdot B_+(t)^{2a-b}\cdot B_-(t)^b\cdot S_+^{2a-2b+2\alpha} e_k\rangle
\intertext{now $2a-2b+2\alpha$ must be zero, i.e., $b = a+\alpha$,}
&= \sum_{a=0}^\infty \frac{1}{(2a)!} \cdot\begin{pmatrix}2a\\ a+\alpha\end{pmatrix}\cdot B_+(t)^{a-\alpha}\cdot B_-(t)^{a+\alpha}
\intertext{summation can start at $a=\alpha$, since otherwise $\left(\begin{smallmatrix} 2a\\ a+\alpha\end{smallmatrix}\right)$ is zero,}
&= \sum_{a=\alpha}^\infty \frac{1}{(2a)!} \cdot\begin{pmatrix}2a\\ a+\alpha\end{pmatrix}\cdot B_+(t)^{a-\alpha}\cdot B_-(t)^{a+\alpha}\\
&= \sum_{a=\alpha}^\infty \frac{B_+(t)^{a-\alpha}\cdot B_-(t)^{a+\alpha}}{(a-\alpha)!\cdot (a+\alpha)!}
\intertext{shift $a$ to start at $a=0$ again}
&= \sum_{a=0}^\infty \frac{B_+(t)^{a}\cdot B_-(t)^{a+2\alpha}}{a!\cdot (a+2\alpha)!}\\
&= \frac{(-B_-(t))^\alpha}{B_+(t)^{\alpha}}\cdot \sum_{a=0}^\infty \frac{(-1)^{a}\cdot |2\cdot B_+(t)|^{2a+2\alpha}}{a!\cdot (a+2\alpha)!\cdot 2^{2a+2\alpha}}\\
&= \left(\frac{B_+(t)}{|B_+(t)|}\right)^{-2\alpha}\cdot J_{|-2\alpha|}(2\cdot |B_+(t)|).
\end{align*}
Hence, the statement is true for $j=-2\alpha$ with $\alpha\in\nset$.

Finally, let $j = -2\alpha-1$ with $\alpha\in\nset_0$. Then
\begin{align*}
&\phantom{=}\;\;\langle e_{k-(2\alpha+1)l},\exp\left( B_+(t)\cdot S_{+} + B_-(t)\cdot S_-\right) e_k\rangle\\
&= \langle S_-^{2\alpha+1} e_k, \sum_{a=0}^\infty \frac{1}{a!} (B_+(t)\cdot S_+ + B_-(t)\cdot S_-)^a e_k\rangle\\
&= \langle e_k, S_+^{2\alpha+1}\cdot \sum_{a=0}^\infty \frac{1}{a!}\cdot \sum_{b=0}^a \begin{pmatrix} a\\ b\end{pmatrix}\cdot B_+(t)^{a-b}\cdot B_-(t)^b\cdot S_+^{a-b}\cdot S_-^b e_k\rangle\\
&= \langle e_k, \sum_{a=0}^\infty \frac{1}{a!}\cdot \sum_{b=0}^a \begin{pmatrix} a\\ b\end{pmatrix}\cdot B_+(t)^{a-b}\cdot B_-(t)^a\cdot S_+^{a-2b+2\alpha+1} e_k\rangle
\intertext{hence $a-2b+2\alpha+1$ must be zero, i.e., summation can be done over all odd $a$'s}
&= \langle e_k, \sum_{a=0}^\infty \frac{1}{(2a+1)!}\cdot \sum_{b=0}^{2a+1} \begin{pmatrix} 2a+1\\ b\end{pmatrix} B_+(t)^{2a+1-b}\cdot B_-(t)^b\cdot S_+^{2a-2b+2\alpha+2} e_k\rangle
\intertext{and therefore $2a-2b+2\alpha+2$ must be zero, i.e., $b = a+\alpha+1$,}
&= \sum_{a=0}^\infty \frac{1}{(2a+1)!}\begin{pmatrix} 2a+1\\ a+\alpha+1\end{pmatrix}\cdot B_+(t)^{a-\alpha}\cdot B_-(t)^{a+\alpha+1}
\intertext{and the summation can start at $a=\alpha$, since otherwise $\left(\begin{smallmatrix}2a+1\\ a+\alpha+1\end{smallmatrix}\right)$ is zero,}
&= \sum_{a=\alpha}^\infty \frac{1}{(2a+1)!}\begin{pmatrix} 2a+1\\ a+\alpha+1\end{pmatrix}\cdot B_+(t)^{a-\alpha}\cdot B_-(t)^{a+\alpha+1}\\
&= \sum_{a=\alpha}^\infty \frac{B_+(t)^{a-\alpha}\cdot B_-(t)^{a+\alpha+1}}{(a-\alpha)!\cdot (a+\alpha+1)!}
\intertext{shifting $a$ to start summation at $a=0$ again we get}
&= \sum_{a=0}^\infty \frac{B_+(t)^{a}\cdot B_-(t)^{a+2\alpha+1}}{a!\cdot (a+2\alpha+1)!}\\
&= \left(\frac{B_+(t)}{|B_+(t)|}\right)^{-2\alpha-1}\cdot \sum_{a=0}^{\infty} \frac{(-1)^a\cdot |2\cdot B_+(t)|^{2a+2a+1}}{a!\cdot (a+2\alpha+1)!\cdot 2^{2a+2\alpha+1}}\\
&= \left(\frac{B_+(t)}{|B_+(t)|}\right)^{-2\alpha-1}\cdot J_{|-2\alpha-1|}(2\cdot |B_+(t)|).
\end{align*}
So finally also the last case $j=-2\alpha-1$ with $\alpha\in\nset_0$ is proved.
\end{proof}

\end{document}